\documentclass[11pt]{amsart}

\usepackage[foot]{amsaddr}

\usepackage{amsmath,amssymb}
\usepackage{mathrsfs}

\usepackage{tikz-cd}

\usepackage[english]{babel}
\usepackage[utf8]{inputenc}
\usepackage[T1]{fontenc}
\usepackage{mathtools}
\usepackage{lmodern}

\usepackage[top=3cm,bottom=2cm,left=3cm,right=3cm,marginparwidth=2.5cm]{geometry}

\usepackage[colorinlistoftodos]{todonotes}
\usepackage[hypertexnames=false]{hyperref}
\hypersetup{
    colorlinks,
    linkcolor={red!50!black},
    citecolor={blue!50!black},
    urlcolor={blue!80!black},
}

\usepackage[shortlabels]{enumitem}
\setlist[1]{itemsep=2pt}
\usepackage[capitalise]{cleveref}

\usepackage{autonum} 

\newcommand{\la}{\langle}
\newcommand{\ra}{\rangle}

\newcommand{\norm}[1]{\left\lVert#1\right\rVert}
\newcommand{\R}{\mathbb{R}}

\newcommand{\N}{\mathbb{N}}
\newcommand{\Z}{\mathbb{Z}}
\newcommand{\B}{\mathrm{Bas}}	
\newcommand{\Lim}{\mathscr{L}}				

\newtheorem{theorem}{Theorem}
\newtheorem{proposition}[theorem]{Proposition}
\newtheorem{corollary}[theorem]{Corollary}
\newtheorem{lemma}[theorem]{Lemma}

\newtheorem{theoremintro}{Theorem}						
\crefname{theoremintro}{Theorem}{Theorems}				

\theoremstyle{definition}
\newtheorem{definition}[theorem]{Definition}
\newtheorem{construction}[theorem]{Construction}

\crefname{construction}{Construction}{Constructions}				

\theoremstyle{remark} 
\newtheorem{remark}[theorem]{Remark}
\newtheorem{example}[theorem]{Example}

\AddToHook{env/lemma/begin}{\crefalias{theorem}{lemma}}
\AddToHook{env/example/begin}{\crefalias{theorem}{example}}
\AddToHook{env/proposition/begin}{\crefalias{theorem}{proposition}}
\AddToHook{env/corollary/begin}{\crefalias{theorem}{corollary}}
\AddToHook{env/definition/begin}{\crefalias{theorem}{definition}}
\AddToHook{env/remark/begin}{\crefalias{theorem}{remark}}
\AddToHook{env/construction/begin}{\crefalias{theorem}{construction}}

\title{Quantitative approximate definable choices}
\author{Antonio Lerario}
\email{\href{mailto:alerario@sissa.it}{lerario@sissa.it}}

\author{Luca Rizzi}
\email{\href{mailto:lrizzi@sissa.it}{lrizzi@sissa.it}}

\author{Daniele Tiberio}
\email{\href{dtiberio:lrizzi@sissa.it}{dtiberio@sissa.it}}

\address{SISSA, via Bonomea 265, 34136 Trieste, Italy}

\date{\today}

\begin{document}
%
\begin{abstract}
In semialgebraic geometry, projections play a prominent role. A \emph{definable choice} is a semialgebraic selection of one point in every fiber of a projection. Definable choices exist by semialgebraic triviality, but their complexity depends exponentially on the number of variables. By allowing the selection to be approximate (in the Hausdorff sense),
 we improve on this result. In particular, we construct an approximate selection whose degree is linear in the complexity of the projection and does not depend on the number of variables. This work is motivated by infinite--dimensional applications, in particular to the Sard conjecture in sub-Riemannian geometry. To prove these results, we develop a general quantitative theory for Hausdorff approximations in semialgebraic geometry, which has independent interest.
\end{abstract}

\maketitle


\section{Introduction}\label{sec:introsemialg}	
	
The goal of this paper is to prove a quantitative approximate version of the \emph{definable choice} theorem, \cite[Thm.\ 4.10]{ComteYomdin}. Before stating our main results, let us explain the name and the context of results of this type.
	
Let $S\subset \R^{n}$ be a compact semialgebraic set, presented as
	\[\label{eq:presS}
	S=\bigcup_{i=1}^a\bigcap_{j=1}^{b_i}\left\{x\in \R^{n}\mid \mathrm{sign}(p_{ij}(x))=\sigma_{ij}\right\},
	\]
	where $\sigma_{ij}\in \{0, +1, -1\}$, the $p_{ij}$ are polynomials of degree at most $d$. We condense this information into the \emph{diagram}  $D(S):=(n, c, d)$ of the representation  \eqref{eq:presS}, where $c=a\max b_i$.
	
	We denote by \[\pi:\R^{n}\to \R^\ell\] the linear projection onto the last $\ell$ coordinates. 
	It follows from Tarski--Seidenberg that the set $\pi(S)$ is semialgebraic and a natural problem is to choose, for every element $y\in \pi(S)$, an element $x(y)\in \pi^{-1}(y)\cap S$ so that the resulting set
	$A:=\{x(y)\}_{y\in \pi(S)}$ is semialgebraic and of dimension at most $\ell$. The set $A$ is called a \emph{definable choice} over $\pi(S)$.

	The fact that this can be done  follows from a result in semialgebraic geometry called \emph{semialgebraic triviality} (see \cref{coro:exist}). However, as a result of this process there is no good control on the geometry of $A$ in terms of the data defining $S$ (see \cref{remark:exponential}). 
		
For many geometric applications one does not really need that $A$ is a choice over $\pi(S)$, but it is enough that it is close to it. More precisely, given $\epsilon>0$, denote by $\mathcal{U}_\epsilon(S)$ the Euclidean $\epsilon$--neighbourhood of $S$. Then one can relax the requirements for the definable choice and ask for a set 
	$A_\epsilon\subseteq \mathcal{U}_\epsilon(S)$,  with $\pi(A_\epsilon)$ ``close'' to $\pi(S)$ (in the Hausdorff metric, denoted by $\mathrm{dist}_H$), and possibly with a control on the diagram of $A_\epsilon$. 
	
In this direction we prove two related results. The first one deals specifically with the problem that we have just discussed (see \cref{thm:appsel}).

\begin{theoremintro}[Quantitative approximate definable choice, first version]\label{thm:appsel-intro} 
For every $c \in\mathbb{N}$ there exist $\kappa\in \N$ such that the following holds. Let $n,\ell,d\in \N$, with $1\leq \ell\leq n$. Let $\pi:\R^{n}\to \R^\ell$ be the projection onto the last $\ell$ coordinates and let $S\subset \R^{n}$ be a bounded closed semialgebraic set with 
\[
D(S)=(n, c, d).
\] 
Then, for every $\epsilon>0$ there exists a closed semialgebraic set $	A_\epsilon\subset \R^{n}$ such that:
	\begin{enumerate}[(i)]
		\item \label{item:intermediate1t-intro} $\dim(A_\epsilon)\leq \ell$;
		\item \label{item:intermediate2t-intro}$A_\epsilon\subseteq \mathcal{U}_\epsilon(S)$;
		\item \label{item:intermediate3t-intro}$\mathrm{dist}_{H}(\pi(A_\epsilon), \pi(S))\leq \epsilon$;
		\item \label{item:intermediate4t-intro} $D(A_\epsilon)=(n, \kappa, \kappa d)$.
	\end{enumerate}
\end{theoremintro}

	The set $A_\epsilon$ from the statement is therefore an approximate (in the Hausdorff metric) definable choice over $\pi(S)$, with a quantitative control on its diagram.
	
	In view of  applications, it is also useful to have the following alternative version of the previous result (\cref{thm:SemialgebraicSelection}), which essentially follows from \cref{thm:appsel-intro} applied to the case of the graph of a semialgebraic map.
	
\begin{theoremintro}[Quantitative approximate definable choice, second version]\label{thm:SemialgebraicSelection-intro}
	For every $ c, d, \ell \in \N$ there exists  $\beta>1$ satisfying the following statement. Let $n\in \N$ and let $K\subset \R^n$ be a closed semialgebraic set contained in the ball $B_{\R^n}(\rho)$, for some $\rho>0$, and let $F :\R^n\to \R^\ell$ be a locally Lipschitz semialgebraic map such that
\begin{equation}
D(\mathrm{graph}(F|_K)) = (n+\ell,c,d).
\end{equation}

Then for every $\epsilon \in (0 , \rho)$ there exists a closed semialgebraic set $C_\epsilon\subset \R^{n}$ such that:
	\begin{enumerate}[(i)]
		\item \label{item:thm1-intro} $\dim(C_\epsilon)\leq \ell$;
		\item  \label{item:thm-intro2} $C_\epsilon \subseteq \mathcal{U}_\epsilon(K)$;
		\item  \label{item:thm3-intro} $\mathrm{dist}_{H}(F(C_\epsilon), F(K))\leq L(F, \rho) \cdot \epsilon$, where $L(F, \rho):=2+\mathrm{Lip}(F, B_{\R^n}(2 \rho))$;
		\item  \label{item:thm4-intro} for every $e=1, \ldots, n$ and every affine space $\R^e\simeq E\subseteq \R^n$, the number of connected components of $E\cap C_\epsilon$ is bounded by
		\begin{equation}\label{eq:boundintroe}
		b_0(E\cap C_{\epsilon})\leq \beta^e.
		\end{equation}
	\end{enumerate}
\end{theoremintro}
\begin{remark}[The case of polynomial maps]
\cref{thm:SemialgebraicSelection} can be applied to any polynomial map $F:\R^n \to \R^\ell$ with components of degree bounded by $d$, assuming that the diagram of the original set satisfies $D(K) = (n,c,d)$. In fact, in this case 
\begin{equation}
\mathrm{graph}(F|_K) = \big\{ (x,y)\in \R^n\times \R^\ell \mid x\in K, \, y=F(x)\big\}
\end{equation} 
is a bounded and closed semialgebraic set with $D(\mathrm{graph}(F|_K))=(n+\ell,c+1,d)$.
\end{remark}

\cref{thm:SemialgebraicSelection-intro} corresponds to  \cite[Thm.\ 4.10 and Ex.\ 4.11]{ComteYomdin}, where it is proved the existence of a semialgebraic set $C_\epsilon$ as in \cref{thm:SemialgebraicSelection-intro}, except for \eqref{eq:boundintroe}, which for them has the shape
\[b_0(E\cap C_{\epsilon})\leq f(n,\ell, c, d, e),\]
for some (non--explicit) function $f:\N^5\to \N$. The most notable conclusion from \cref{thm:SemialgebraicSelection-intro}  is therefore  the explicit dependence of the bound on the dimension of the affine space of \cref{item:thm4-intro}, which says that we can take $f(n, \ell, c, d, e)=\beta(c, d, \ell)^e$.

Similarly, the difficult part from \cref{thm:appsel-intro} is proving that the diagram of $A_\epsilon$  has the explicit shape $D(A_\epsilon)=(m, \kappa, \kappa d)$ with $\kappa$ depending only on the combinatorial data $c$ of the diagram of the original set (and not on the number of variables, for instance).

Obtaining this explicit dependence is non--trivial. Compared to \cite[Thm.\ 4.11]{ComteYomdin}, a conceptual novelty is the use of ideas from \cite{BasuLerario}, which in turn involves the study of approximation of semialgebraic sets in the Hausdorff metric.  A large part of the paper is devoted to expanding and developing these ideas, see \cref{sec:Hsa}. 
\subsection{Hausdorff approximations}\cref{sec:Hsa} contains new ideas that have their own interest and that are based on the observation that the Hausdorff distance between semialgebraic sets in $\R^n$ can be studied  within the framework of semialgebraic geometry. In particular, this notion can be defined also on a real closed extension $R$ of the real numbers (see \cref{sec:semialg}), leading to a notion of ``Hausdorff distance'' between semialgebraic sets in $R^n$. The main result from \cref{sec:Hsa} is then a technique to produce Hausdorff approximations of semialgebraic sets using infinitesimals, i.e. working in the real closed field of algebraic Puiseux series (see \cref{def:puiseux}). In the current paper we allow multiple infinitesimals, which makes the technique handful and practical, but which requires nontrivial  extensions of the ideas from \cite{BasuLerario}. What is important for us is that this technique allows to keep control on the combinatorial part and the degrees of the approximating set: for example, in \cref{propo:approxclosed}, we show how to approximate a closed semialgebraic set with a \emph{closed basic} semialgebraic set (see \cref{def:semialgebraic}) whose combinatorial data and degree are controlled in terms of the diagram of the original set (see the beginning of \cref{sec:hacb} to appreciate the subtlety of this statement).

\subsection{Infinite--dimensional applications}\cref{thm:SemialgebraicSelection-intro} was recently used in  \cite{LRT-Sard}, where we study the Sard property for polynomial maps in \emph{infinite} dimension. 
The starting point for \cite{LRT-Sard} is a quantitative proof given by Yomdin of the classical Sard's theorem, in \emph{finite} dimension, that uses ideas from semialgebraic geometry \cite{ComteYomdin}. 

To explain the key point, recall that for a $C^1$ map $F:\R^n\to \R^m$, and $\Lambda=(\Lambda_1, \ldots, \Lambda_m)\in \R^m_+$ the set of almost--critical values of $F$ is defined by
	\begin{equation}
	C^{\Lambda}(F):=
	\bigg\{ x \in \R^n\,\bigg| \,
	\sigma_i(D_x F ) 
	\leq \Lambda_i, 
	\quad
	\forall\, i=1,\dots,m
	\bigg\},
	\end{equation}
	where $\sigma_1(D_x F)\geq \cdots \geq \sigma_m(D_xF)$ are the singular values of $D_xF$ (here we assume $n\geq m$). In the finite--dimensional case, the study of the measure of the critical values of $F$ can be reduced to estimating the ``size'' of the set of its almost critical values, i.e. of the image of its almost critical points. For doing this, one can use the theory of \emph{variations}, introduced by Vitushkin \cite{Vit1}
	 and developed in \cite{ComteYomdin}.  (Recall that the $i$--th variation of a set $S$, denoted by $V_i(S)$ is a sort of $i$--dimensional volume of $S$.)
	
	In \cite[Cor.\ 7.4]{ComteYomdin} a quantitative estimate on the variations of the almost--critical values of polynomial maps has been obtained. In that estimate the dimensional parameter $n$ does not appear explicitly. Using \cref{thm:SemialgebraicSelection-intro}, in \cite{LRT-Sard} we obtain the following result, which makes this dependence explicit.
	
	\begin{theorem}[{\cite[Theorem G]{LRT-Sard}}]\label{thm:estimatevariations-intro}
	Let $n\geq m$, and $p: \R^n \to \R^m$ be a polynomial map with components of degree $\leq d$. For  $ i= 0, \dots, m$, $\Lambda=(\Lambda_1 , \dots , \Lambda_m)\in \R^m_+$ and $\rho >0$, we have  
				\begin{equation}\label{eq:variations}
		V_i( p ( C^{ \Lambda }( p ) \cap B_{\R^n}(\rho) ))
			\leq
			\mathrm{cst}(m,\rho)n^m			\beta_0^n \Lambda_{0} 
			\cdots \Lambda_{i},
		\end{equation}
					where $\beta_0=\beta_0(d,m)$ depends only on $d$ and $m$, $\mathrm{cst}(m,\rho)$ depends only on $m,r$, and  $\Lambda_0:=1$.
	\end{theorem}	
The polynomial dependence on $n$ of \eqref{eq:variations} is a consequence of \cref{item:thm4-intro} from \cref{thm:SemialgebraicSelection-intro}. Letting $n\to \infty$  this explicit estimate is what is used in \cite{LRT-Sard} to prove Sard-type results for suitable maps $f:H\to \R^m$, where $H$ is an infinite--dimensional Hilbert space.

\subsection{Acknowledgements}

This project has received funding from (i) the European Research Council (ERC) under the European Union's Horizon 2020 research and innovation programme (grant agreement GEOSUB, No.\ 945655); (ii) the PRIN project ``Optimal transport: new challenges across analysis and geometry'' funded by the Italian Ministry of University and Research; (iii) the Knut and Alice Wallenberg Foundation. The authors also acknowledge the INdAM support.
The authors wish to thank Saugata Basu for helpful discussions.

\section{Hausdorff approximations in semialgebraic geometry}\label{sec:quantitativesemialgebraic}

\subsection{Semialgebraic sets and maps}\label{sec:semialg} 
In order to study Hausdorff approximations of semialgebraic sets in $\R^n$, it will be convenient to work with an extension $R$ of the field of real numbers, where real ``infinitesimals'' will be themselves elements of the field (below we will take for $R$ the field of 
algebraic Puiseux series with coefficients in $\mathbb{R}$). For this reason, in this section we recall the basic notions from semialgebraic geometry over a general \emph{real closed field} and we refer the reader to the monographs \cite{BasuPoRoyBook, BCR} for more details.
\begin{remark}[On the use of the language of real closed fields]The language of real closed fields might be unfamiliar for some readers, but it is especially useful in this context. We want to stress that using this language is not strictly necessary and it is possible that the proofs below can be formulated only using semialgebraic geometry in $\R^n$. However, it has become quite standard in semialgebraic geometry since, once it is introduced, the proofs and the statements become much shorter. Moreover, we are quoting some technical constructions from \cite{BasuLerario, BasuPoRoyBook, BR} that are stated using this language, and translating them to the classical semialgebraic language would make this part more technical -- the precise point where Puiseux series will be used instead of classical real algebraic geometry is \cref{prop:5properties}. From here, the tools of Hausdorff approximation that we develop in \cref{sec:Hsa} will allow to get back to real sets, keeping control of this process.
\end{remark}

Recall that a real closed field $R$ is an ordered field whose positive cone is the set of squares of elements from $R$, and such that every polynomial in $R[x]$ of odd degree has a root in $R$. The order of $R$ allows to define the sign of an element $r\in R$ as:
\begin{equation}
\mathrm{sign}(r):=\begin{cases}+1&r>0,\\
-1&r<0,\\
\,\,0&r=0.
\end{cases}
\end{equation}
We denote by $R_{+}:=\{r\in R\,|\, \mathrm{sign}(r)\geq 0\}$ the cone of non-negative elements of $R$.

The field $\R$ is real closed. The main example of real closed field we use is the following.

\begin{definition}[Algebraic Puiseux series]\label{def:puiseux}
Let $R$ be a real closed field and $R(\zeta)$ be the field of rational functions in the variable $\zeta$. The real closed field $R\la\zeta\ra$  of \emph{algebraic Puiseux series with coefficients in $R$} is defined by:
\[ 
R\la\zeta\ra:=\left\{f=\sum_{k=k_0}^{\infty} a_k\zeta^{\frac{k}{q}}\,\bigg|\,a_k\in R,\, k_0\in \Z,\,q\in \N,\, \textrm{$f$ is algebraic over $R(\zeta)$}\right\}.
\]
\end{definition}
In order to define a ring structure on the set $R\la \zeta\ra$, we insist that for $r_1, r_2\in \mathbb{Q}$ it holds
\[
\zeta^{r_1}\zeta^{r_2}=\zeta^{r_1+r_2}, \quad (\zeta^{r_1})^{r_2}=\zeta^{r_1 r_2},\quad\zeta^0=1.
\]
Therefore, two Puiseux series $a=\sum_{k\geq k_1}a_k\zeta^{k/q_1}$ and $b=\sum_{k\geq k_2}b_k\zeta^{k/q_2}$ can be written as a formal power series in $\zeta^{1/q}$, where $q$ is the least common multiple between $q_1$ and $q_2$. This allows to add and multiply Puiseux series. If $a=a_1\zeta^{r_1}+a_2\zeta^{r_2}+\cdots\in R\la \zeta\ra,$ with $a_1\neq 0$ and $r_1<r_2<\cdots$, we say that $a>0$ if $a_1>0.$ From this definition it follows that $0<\zeta<r$ for every $r\in R$ with $r>0$, and this is why $\zeta$ is called \emph{infinitesimal}.

\begin{remark}[Algebraic Puiseux series as germs]\label{rem:iso}
The field $R\la \zeta\ra$ is isomorphic, as a real closed field, to the field of continuous semialgebraic functions germs $f:(0, \delta)\to R$, where $(0, \delta)$ is an interval in $R$, \cite[Thm.\ 3.17]{BasuPoRoyBook}. Addition and multiplication translates to the usual ones. The order is defined as follows: the germ of a semialgebraic function $f:(0, \delta)\to R$ is positive if and only if there exists $0<\delta'<\delta$ such that $f(\zeta)>0$ for every $\zeta\in (0, \delta')$. Given a germ of a continuous semialgebraic function $f:(0, \delta)\to R$, in order to get an expression like $f=\sum_{k\geq k_0}a_k\zeta^{k/q}$, i.e. the corresponding element in the field of algebraic Puiseux series, one uses the fact that the graph of $f$ lies on a \emph{branch} of an algebraic curve $P(x, y)=0$ in the plane $R^2$, where $P$ is a polynomial with coefficients in $R$. In particular, this branch admits a parametrization near the origin as 
\[
x(\tau)=\tau^q,\quad y(\tau)=\sum_{k\geq k_0}a_k \tau^{k},
\]
where $q\in \mathbb{N}, k_0\in \Z$ and $a_k\in R$, \cite[Thm.\ 2.2]{Walker}. This means that $P(x(\tau), y(\tau))=0$ as formal power series, and one gets a Puiseux series for $f$ by the substitution $\tau=\zeta^{1/q}$ in $y(\tau)$.
\end{remark}

If $R$ is a real closed field, then so is $R\la \zeta\ra$, see \cite[Cor.\ 2.98]{BasuPoRoyBook}. Using this observation, letting $K=R\la\zeta_1\ra$, then also $R\la \zeta_1, \zeta_2\ra=K\la\zeta_2\ra$ is a real closed field. Repeating this argument inductively, we define the real closed field of \emph{algebraic Puiseux series with multiple infinitesimals and coefficients in $R$} as
\[R\langle \zeta_1, \ldots, \zeta_m\rangle:=R\langle \zeta_1\rangle\cdots \langle \zeta_m\rangle.
\]

Let $R$ be a real closed field.  The \emph{norm} of an element $x=(x_1, \ldots, x_n)\in R^n$ is defined as \[\|x\|:=\sqrt{x_1^2+\cdots +x_n^2}\in R.\] For every $x=(x_1, \ldots, x_n)\in R^n$ and $r\in R_{+}$, we denote by ${B}_{R^n}(x,r)\subset R^n$ the set
\begin{equation}
B_{R^n}(x,r):=\left\{y\in R^n\,\bigg|\, \|x-y\|\leq r\right\},
\end{equation}
and call it the \emph{closed ball around $x$ of radius $r$}. The \emph{Euclidean topology} on $R^n$ is the topology generated by the closed balls. These notions are inspired by their classical counterparts in the case $R=\R$,  with the caveat that in these definitions $r$ is not necessary real. Similarly, the order of a real closed field $R$ can be used to define the notion of semialgebraic sets.

\begin{definition}[Semialgebraic sets and maps]\label{def:semialgebraic}
We say that a set $S\subset R^n$ is \emph{semialgebraic} if
\begin{equation}\label{eq:defA} 
S=\bigcup_{i=1}^a\bigcap_{j=1}^{b_i}\left\{x\in R^n\mid \mathrm{sign}(p_{ij}(x))=\sigma_{ij}\right\},
\end{equation}
for some finite set of polynomials  $p_{ij}\in R[x_1, \ldots, x_n]$ and $\sigma_{ij}\in \{0, +1, -1\}.$
We call the description \eqref{eq:defA} a \emph{representation} of $S$.

Given two sets of polynomials $\mathcal{P}=\{p_1, \ldots, p_{a_1}\},\mathcal{Q}=\{q_1, \ldots, q_{a_2}\}\subset R[x_1, \ldots, x_n]$, we define the  \emph{algebraic} set $Z(\mathcal{P};R)$  and  the \emph{closed basic} semialgebraic set $\B(\mathcal{P}, \mathcal{Q};R)\subset R^n$ by
\begin{equation}
Z(\mathcal{P};R):=\{x\in R^n\mid p_1(x)=\cdots=p_{a_1}(x)=0\},
\end{equation}
\begin{equation}
 \B(\mathcal{P}, \mathcal{Q}; R):=Z(\mathcal{P};R)\cap \left(\bigcap_{j=1}^{a_2}\left\{x\in R^n\mid q_j(x)\leq 0\right\}\right).
 \end{equation}
 
Finally, a map $f:A\to B$ between semialgebraic sets $A\subset R^n, B\subset R^\ell$ is said to be \emph{semialgebraic} if its graph is a semialgebraic set in $R^n\times R^\ell$.
\end{definition}

The representation of a semialgebraic set $S$ as in \eqref{eq:defA}  is not unique, however having such a representation quantifies the complexity of $S$, using the following notion.
\begin{definition}[Diagram of a semialgebraic set]\label{def:diagram}
Let $S\subset R^n$ be a semialgebraic set represented  as in \eqref{eq:defA}. We say that the triple
\begin{equation}\label{eq:diagram}
\left(n, a\cdot\max_i\{{b_i}\}, \max_{i,j}\{\deg(p_{ij})\}\right)\in \N^3
\end{equation}
is a \emph{diagram} for $S$. Below, the equation ``$D(S)= (m, c, d)$'' will mean that there exists a representation of $S$ as in \eqref{eq:defA} with $n\leq m$, $ a\cdot\max_i\{{b_i}\}\leq c$ and $\max_{i,j}\{\deg(p_{ij})\}\leq d$.
\end{definition}

It is often useful to use an alternative (equivalent) description of semialgebraic sets, using the notion of \emph{first--order formulas}.

\begin{definition}[First--order formula]\label{def:firstorder}
Let $R$ be a real closed field. A first--order formula of the language of ordered fields with coefficients in $R$ is a formula written with a finite number of conjunctions, disjunctions, negations, and universal or existential quantifiers on variables, starting from atomic formulas which are formulas of the kind $p(x_1, \ldots, x_n) = 0$ or $q(x_1, \ldots ,x_n) < 0$, where $p$ and $q$ are polynomials with coefficients in $R$. The free variables of a formula are those variables of the polynomials appearing in the formula which are not quantified.
\end{definition}

Semialgebraic sets are precisely those defined by first--order formulas \cite[Prop.\ 2.2.4]{BCR}.
\begin{definition}[Set defined by a formula, see {\cite[Sect.\ 1.1]{BasuPoRoyBook}}]
Let $\psi$ be a first--order formula of the language of ordered fields, with coefficients in $R$, and with $n$ free variables. The \emph{set defined by} $\psi$ in $R^n$ (or the \emph{realization} of $\psi$ in $R^n$) is the semialgebraic set
\[\label{eq:Sformula}
 \mathrm{Reali}(\psi; R)\subseteq R^n
 \]
defined by induction on the construction of the formula, starting from atoms:
\[\mathrm{Reali}(p =0;R):=\{x\in R^n\,|\, p(x)=0\},\quad \mathrm{Reali}(p< 0;R) := \{x \in R^n \,|\, p(x)< 0\},\]
($p$ is a polynomial with coefficients in $R$), 
\begin{gather}
\mathrm{Reali}(\phi_1 \land \phi_2; R) := \mathrm{Reali}(\phi_1; R) \,\cap\, \mathrm{Reali}(\phi_2; R),\\
\mathrm{Reali}(\phi_1 \lor \phi_2; R) := \mathrm{Reali}(\phi_1; R)\, \cup\, \mathrm{Reali}(\phi_2; R),\\
\mathrm{Reali}( \lnot\phi; R ) := R^n \setminus \mathrm{Reali}(\phi; R ),\\
\mathrm{Reali}((\exists y) \,\phi; R) := \{x\in R^n \,|\, \exists y\in R,\, (x, y) \in \mathrm{Reali}(\phi; R)\},\\
\mathrm{Reali}((\forall y)\, \phi; R) := \{x \in R^n\,|\,  \forall y\in R,\, (x, y) \in  \mathrm{Reali}(\phi;R)\},
\end{gather}
where $\phi_1, \phi_2, \phi$ are first--order formulas with an appropriate number of free variables.
\end{definition}

\subsection{Some properties of semialgebraic sets}\label{sec:useful}

We collect here some properties of  semialgebraic sets over a real closed field.

\subsubsection{Semialgebraic triviality}\label{item2useful}
Continuous semialgebraic maps $f:A\to B$ are ``piecewise'' trivial fibrations: there exists a partition of $B$ into finitely many semialgebraic sets
\begin{equation}
B=\bigsqcup_{j=1}^b B_j
\end{equation}
and, for every $j=1, \ldots, b$ there exist fibers $F_j:=f^{-1}(y_j)$, for some $y_j\in B_j,$ and a semialgebraic homemorphism $\varphi_j:B_j\times F_j\to f^{-1}(B_j) $ that makes the following diagram commutative:
\begin{equation}
\begin{tikzcd}
B_j\times F_j \arrow[rr, "\varphi_j"] \arrow[rd, "p_1"'] &     & f^{-1}(B_j) \arrow[ld, "f"] \\
                                                      & B_j &                            
\end{tikzcd}
\end{equation}
This result is called \emph{semialgebraic triviality}, see \cite[Thm.\ 9.3.2]{BCR}. 

\begin{corollary}\label{coro:exist}
Definable choices exist.
\end{corollary}
\begin{proof}
Let $S\subset \R^{n}$ be a semialgebraic set and $\pi:\R^n\to \R^\ell$ the projection. Then, by semialgebraic triviality there exists a finite partition of
$\pi(S)=\sqcup_{j=1}^b B_j$ into semialgebraic sets and for every $j=1, \ldots, b$ there exists fibers $F_j=\pi^{-1}(y_j)\subset S$, with $y_j\in B_j$, and semialgebraic homeomorphisms $\varphi_j:B_j\times F_j\to \pi^{-1}(B_j)\cap S$ making the corresponding diagram commutative. For every $j=1, \ldots, b$ choose an element $x_j\in F_j$. The definable choice of $S$ over $\pi(S)$ is the set
	\begin{equation}
\bigcup_{j=1}^b\varphi_j
\left( B_j \times \{x_j\}  \right),
	\end{equation} 
	concluding the proof.
\end{proof}
\begin{remark}\label{remark:exponential}
Note that the complexity of a set built as in the proof of \cref{coro:exist} depends on the complexity of the objects involved in the semialgebraic triviality. This uses the so--called cylindrical algebraic decomposition and it is known to be doubly exponential in the number of variables of $S$ (see \cite{davenportheintz, browndavenport}).
\end{remark}

\subsubsection{Dimension and stratifications}\label{item1useful}
 Every semialgebraic subset of $R^n$ can be written as a finite union of semialgebraic sets, each of them semialgebraically homeomorphic to an open cube $(0,1)^m\subset R^m$, for some $m \leq n$ (\cite[Thm.\ 2.3.6]{BCR}). This allows to define the \emph{dimension} of a semialgebraic set as the maximum of the dimensions $m$ of these cubes. In fact, it is possible to introduce the notion of smoothness also over general real closed fields, and every semialgebraic set can be written as a finite union of smooth, semialgebraic disjoint manifolds called \emph{strata} (\cite[Prop.\ 9.1.8]{BCR}). This is called a \emph{Nash stratification}. We will use this last result only in the classical case $R=\R.$

The dimension of a semialgebraic set is preserved by semialgebraic homeomorphisms and behaves naturally under Cartesian product: 
\begin{equation}\label{eq:dimension}\dim(A\times B)=\dim(A)+\dim(B).\end{equation}
Moreover, if $f:A\to B$ is a continuous semialgebraic map, then
\begin{equation}\label{eq:dimension2}
\dim(f(A))\leq \dim(A).
\end{equation}

We will also need a stronger notion of dimension, introduced in \cite{BR}.
\begin{definition}[Strong dimension]\label{def:stronglydimk}
Let $S\subset R^{n}$ be a semialgebraic set and $1\leq \ell\leq n$. We say that $S$ is \emph{strongly of dimension $\leq \ell$} if, letting $\pi_\ell: R^{n} \to R^\ell$ be the projection on the last $\ell$ coordinates, it holds
\[\label{eq:stronglydim-def}
\forall y\in R^\ell\quad \#\left(S \cap \pi_{\ell}^{-1}(y)\right)<\infty.
\]
\end{definition}
It follows from semialgebraic triviality (see \cref{item2useful}) that if \eqref{eq:stronglydim-def} holds, then
\[\label{eq:stronglydim-def2}
\sup_{y\in R^\ell}  \#\left(S \cap \pi_{\ell}^{-1}(y)\right)<\infty,
\]
so that in \cref{def:stronglydimk} one can replace \eqref{eq:stronglydim-def} with the apparently stronger condition \eqref{eq:stronglydim-def2}. 

Moreover, if $S$ is strongly of dimension $\leq \ell$, then $\dim(S)\leq \ell$. However, a semialgebraic set $S$ of dimension $\ell$ may not be strongly of dimension $\leq \ell$, as this depends on the relative position between $S$ and the projecting subspace $R^\ell$. Nevertheless we have the following result.

\begin{proposition}[Genericity of strong dimension]\label{propo:mather}
Let $S\subset R^{n}$ be a semialgebraic set of dimension $\ell$. The set of invertible linear transformations $L:R^{n}\to R^{n}$ such that $L(S)$ is strongly of dimension $\leq \ell$ is semialgebraic and dense.
\end{proposition}
\begin{proof}
We first prove the case $R=\R$. The set of such linear transformations is semialgebraic, as it admits a semialgebraic description in the coefficients of $L$.

To prove density, let $M\subset \R^{n}$ be a smooth $p$--dimensional manifold with $p\leq \ell$. Given $r\in \N$ and a Thom--Boardman manifold $\Sigma\subset J^{r}(M, \R^\ell)$, it follows from \cite[Thm. 2]{Mather} that the set of linear maps $T:\R^{n}\to \R^\ell$ that restrict to maps $T|_M:M\to \R^\ell$ which are transversal to $\Sigma$  is dense. In particular, the set of linear maps $T:\R^{n}\to \R^\ell$ such that $T|_{M}$ is a stratified immersion is dense. The fibers of every such map are discrete and every such map can be obtained as $T=\pi_\ell\circ L$ for an appropriate $L:\R^{n}\to \R^{n}$.  Therefore, for every smooth $p$--dimensional manifold $M\subset \R^{n}$, with $p\leq \ell$, the set  ${\mathcal{L}}(M)$ of linear maps $L:\R^{n}\to \R^{n}$ such that the fibers of the map $\pi_\ell|_{L(M)}:L(M)\to \R^\ell$ are discrete is dense. When $M$ is semialgebraic, the set ${\mathcal{L}}(M)$ is also semialgebraic and therefore it contains an open dense set.

Let  $S=\sqcup_{i=1}^a M_i$ be a Nash stratification of $S$. Applying the above argument to each stratum $M_i$, we see that the set ${\mathcal{L}}(S):={\mathcal{L}}(M_1)\cap\cdots \cap {\mathcal{L}}(M_a)$ is semialgebraic and contains an open dense set. For every $L\in {\mathcal{L}}(S)$ the fibers of $\pi_\ell|_{L(S)}:L(S)\to \R^{\ell}$ are discrete and semialgebraic, therefore they are finite. Hence $L(S)$ is strongly of dimension $\leq \ell$. This concludes the proof for the case $R=\R$. 

The statement we just proved can be written as a first--order formula with coefficients in $\R$. Therefore, the same statement holds if $\R$ is replaced by a real closed extension $R$ of it, by  \cite[Prop.\ 5.2.3]{BCR}. (This is the so-called \emph{transfer principle} from real algebraic geometry:  every property that can be expressed in the first--order language of ordered fields with coefficients in $R$ can be transferred to any real closed extension of $R$.)
\end{proof}

%
\subsubsection{Thom-Milnor bound}\label{item:milnorbound}
For every $c\in \N$ there exists $\beta_{\mathrm{tm}}=\beta_{\mathrm{tm}}(c)>1$ such that, if $S$ is a semialgebraic set with diagram $D(S)=(n, c, d)$, then the number of connected components of $S$, denoted by $b_0(S)$, can be bounded by (\cite[Thm.\ 7.50]{BasuPoRoyBook}):
\begin{equation}\label{eq:milnorbound}
b_0(S)\leq \beta_{\mathrm{tm}}d^n.
\end{equation}

\subsection{Hausdorff approximations using infinitesimals}\label{sec:Hsa}

Using Puiseux series we can produce Hausdorff approximations of semialgebraic sets in $\R^n$ as ``specializations'' of semialgebraic sets defined on an extension $R^n$ which contains some ``infinitesimals'', and this can be done in a controlled way. Let us explain this idea.

\begin{example}\label{example:neighbourhood}Given a set $A\subset \R^n$, for $r>0$ denote its $r$--neighbourhood in $\R^n$ by
\begin{equation}\label{eq:neighbourhood}
\mathcal{U}_r(A):=\left\{x\in \R^n\,\bigg|\, \exists a=(a_1, \ldots, a_n)\in A \quad \text{s.t.} \quad \sum_{i=1}^n(a_i-x_i)^2 \leq r^2\right\}.
\end{equation}
If $A$ is semialgebraic, defined by a first order formula $\phi$, then $\mathcal{U}_r(A)$ is defined by the first order formula 
\begin{equation}
\phi_r:=\left(\exists a \left( \phi(a) \land  \sum_{i=1}^n(a_i-x_i)^2\leq r^2\right)\right).
\end{equation}
Instead of interpreting $\{\phi_r(x)\}_{r>0}$ as a \emph{family} of first order formulas with coefficients in $\R$, we can interpret it as a single first order formula $\psi$ with coefficients in $R=\R\la \zeta\ra$
\begin{equation}\label{eq:neighbzeta}
\psi:=\left(\exists a \left( \phi(a) \land \sum_{i=1}^n(a_i-x_i)^2\leq \zeta^2\right) \right).
\end{equation}
The set $S\subset R^n$ defined by the formula $\psi$ encodes the family of neighbourhoods of the original semialgebraic set $A$ in a semialgebraic way. Note that the $r$--neighbourhood of $A$ is described by the  formula obtained from $S$  by ``evaluating'' the infinitesimal $\zeta$ at $r$. 
\end{example}

\subsubsection{The evaluation map and the map ${\Lim}_0$}
 The procedure described in \cref{example:neighbourhood} is a special case of a general construction of ``evaluation'' of a semialgebraic set of Puiseux series. First, we provide the following definition, inspired by \cite[Def.\ 1.7]{gavo}, to formalize the concept of properties that are true for ``sufficiently small'' nested sequences of Puiseux series. The intricacy of the statements for the case $m>1$ is a consequence of the iterative definition of Puiseux series $R\la\zeta_1,\dots,\zeta_m\ra$ for several infinitesimals.

\begin{definition}[Predicates for sufficiently small Puiseux series]\label{def:suffsmall}
Let $R$ be a real closed field, $1\leq k\leq m$, and let $\mathscr{P}=\mathscr{P}(t_k, \dots, t_m)$ be a property where $t_j\in R\la\zeta_1, \ldots, \zeta_{j-1}\ra$ for $k\leq j\leq m$. We say that the property $\mathscr{P}$ holds for \[0<t_k\ll\cdots \ll t_{m}\ll1\]
if there exists $\delta_m\in R\la \zeta_1, \ldots, \zeta_{m-1}\ra$ with $\delta_m>0$ such that for $t_m\in R\la \zeta_1, \ldots, \zeta_{m-1}\ra$ with $0<t_m<\delta_m$ there exists $\delta_{m-1}\in R\la \zeta_1, \ldots, \zeta_{m-2}\ra$ with $\delta_{m-1}>0$ such that for all $t_{m-1}\in R\la \zeta_1, \ldots, \zeta_{m-2}\ra$ with $0<t_{m-1}<\delta_{m-1}$ there exists (...) $\delta_k\in R\la \zeta_1, \ldots, \zeta_{k-1}\ra$ with $\delta_k>0$ such that for all $t_k\in R\la \zeta_{1}, \ldots \zeta_{k-1}\ra$ with $0<t_k<\delta_k$ we have $\mathscr{P}(t_k, \ldots, t_m)$.
\end{definition}
\begin{definition}[The evaluation map]\label{def:Lt}
Let $R$ be a real closed field and let $\psi$ be a formula with $n$ free variables and with coefficients in $R\la\zeta_1, \ldots, \zeta_m\ra$. For every $1\leq k\leq m$ and for every $(t_k, \ldots, t_m)$ with $t_j\in R\la\zeta_1, \ldots, \zeta_{j-1}\ra$ for $j=k, \ldots, m$, we denote by $\psi|_{\zeta_k=t_k, \ldots, \zeta_m=t_m}$ the formula with coefficients in $R\la\zeta_1, \ldots, \zeta_{k-1}\ra$ that is obtained from $\psi$ by replacing iteratively in its coefficients each $\zeta_j$  by $t_j$, starting from $j=m$. This can be done provided that, at each step, $t_j$ is sufficiently small so that the coefficients of the formula coming from the previous step (which are finitely many and which we see as germs of continuous semialgebraic functions in $\zeta_j$) have a representative defined on a common interval containing $t_j$, yielding a Puiseux series in $j-1$ infinitesimals. If $S\subset R\langle\zeta_1, \ldots, \zeta_m\rangle^n$ is the semialgebraic set defined  by $\psi$, we denote by $S|_{t_k, \ldots, t_m}\subset R\la\zeta_1, \ldots, \zeta_{k-1}\ra^n$ the semialgebraic set defined by the formula $\psi|_{\zeta_k=t_k, \ldots, \zeta_m=t_m}$, i.e.
\[\label{eq:defS_t}
S|_{t_k, \ldots, t_m}:=\mathrm{Reali}\bigg(\psi|_{\zeta_k=t_k, \ldots, \zeta_m=t_m}; R\la\zeta_1, \ldots, \zeta_{k-1}\ra\bigg)\subset R\la\zeta_1, \ldots, \zeta_{k-1}\ra^n,
\]
provided that the right hand side of \eqref{eq:defS_t} is defined. In particular, this is the case for $0<t_k\ll\dots \ll t_{m}\ll1$. The set $S|_{t_k, \ldots, t_m}$ is called the \emph{evaluation} of $S$ (at $t_k,\dots,t_m$).
\end{definition}
\begin{remark}\label{remark:L1}
Indeed, \eqref{eq:defS_t} depends on the formula $\psi$ defining $S$; for us the formula will be clear and this abuse of notation will create no harm.
\end{remark}
\begin{remark}[Polynomial coefficients]\label{remark:L2}
If the coefficients of the formula $\psi$ are \emph{polynomials} in the infinitesimals (and not general algebraic Puiseux series), then \eqref{eq:defS_t} is defined for any $(t_k,\dots,t_m)$ with $t_j \in R[\zeta_1,\dots,\zeta_{j-1}]\subset R\la\zeta_1,\dots,\zeta_{j-1}\ra$. In this case, at the $i$--th step of the process one gets a new formula with coefficients in $R[\zeta_1, \ldots, \zeta_{m+i-2}]$; therefore the process can be iterated and the concept of ``sufficiently small'' from \cref{def:suffsmall} can be dispensed of. This is the approach followed in \cite{BasuLerario} in the case of one infinitesimal ($m=1$). Our general formulation, albeit defined only for $0<t_k\ll\dots\ll t_m\ll 1$, is more flexible and technically necessary. Notably, the key \cref{lemma:cardinalitaFormula-one,lemma:cardinalitaFormula} require evaluation at general Puiseux series.
\end{remark}

\begin{remark}[Evaluations and germs]\label{remark:distinct}
Let $S=\mathrm{Reali}(\psi;R\la \zeta\ra)\subset R\la \zeta\ra^n$ be a semialgebraic set. Recall from \cref{rem:iso} that the germ of a continuous semialgebraic function $g:(0, \delta)\to R$ can be identified to an element of $R\la \zeta\ra$ (we denote this germ still by $g$). Then, using the notation from \cref{def:Lt}, the following property is true by construction:
\[\label{property}
\exists \delta>0\quad \forall r\in (0,\delta)\quad g(r)\in  S|_r \quad \iff g\in S.
\]
\end{remark}	

\begin{lemma}[Evaluation preserves the diagram]\label{lem:diagramext}
In the same setting of \cref{def:Lt}, for every $1\leq k\leq m$  it holds
\begin{equation}\label{eq:diagramext}
D(S|_{t_k,\dots,t_m})= D(S),
\end{equation}
for all $(t_k,\dots,t_m)$ such that the evaluation is defined.
\end{lemma}
\begin{proof}
The statement follows from the fact that a representation of $S|_{t_1,\dots,t_m}$ is obtained from a representation of $S$ by replacing $\zeta_j$ with $t_j$ in the coefficients, in the prescribed order. 
\end{proof}

{ The strong dimension condition of \cref{def:stronglydimk} is not increased by small evaluations.
\begin{lemma}[Small evaluations do not increase the strong dimension: one infinitesimal]\label{lemma:cardinalitaFormula-one}
In the same setting of \cref{def:Lt}, with $m=1$. Let $\pi_\ell : R\la\zeta\ra^{n}\to R\la\zeta\ra^\ell$ be the projection on the last $\ell$ coordinates, $1\leq \ell\leq n$. Then, there exists $\delta \in R$, $\delta>0$ such that for all $t \in R$ with $0<t<\delta$ it holds
\[\label{eq:monoS}
\sup_{y \in R^\ell} \# S|_{t}\cap\pi_\ell^{-1}(y) \leq \sup_{w \in R\la \zeta\ra^\ell} \# S\cap\pi_\ell^{-1}(w).
\]
In particular, if $S\subset R\la \zeta\ra^n$ is strongly of dimension $\leq \ell$, then the set $S|_{t}$ is also strongly of dimension $\leq \ell$.
\end{lemma}
\begin{proof}
In the proof, set $\pi = \pi_\ell$. Let
\[\label{eq:supN}
N:=\sup_{w\in R\la\zeta\ra^\ell}\#\left(S\cap \pi^{-1}(w) \right).
\]
If $N=\infty$ there is nothing to prove, then assume $N<\infty$. 

Assume by contradiction that for all $\delta\in R$, $\delta>0$, there exists $t\in R$, $0<t<\delta$, and $y\in R^\ell$ such that
\[
\#S|_t\cap \pi^{-1}(y) >N.
\]
Consider the semialgebraic set $D\subset R^\ell\times R_+$ defined by
\[
D:=\left\{(y,t)\in R^\ell\times R_+\mid \# S|_t \cap \pi^{-1}(y) > N\right\}.
\]
Let $P_2: D \to R_+$ be the projection on the second factor. By our assumption the image $P_2(D)\subset R_+$ contains a sequence accumulating to $0$. By semialgebraic triviality (\cref{item2useful}) there exists $\delta_0\in R_+$, $\delta_0>0$, a point $t_0\in (0,\delta_0)$ a fiber $F_0 = P_2^{-1}(t_0)$ and a semialgebraic homeomorphism $\Phi: (0,\delta_0)\times F_0 \to P_2^{-1}((0,\delta_0))$ such that the following diagram is commutative
\begin{equation}
\begin{tikzcd}
(0,\delta_0)\times F_0 \arrow[rr, "\Phi"] \arrow[rd, "p_1"'] &     & P_2^{-1}((0,\delta_0)) \arrow[ld, "P_2"] \\
                                                      & (0,\delta_0) &                            
\end{tikzcd}
\end{equation}
Fix an element $f_0\in F_0$ and consider the function $y_0:(0,\delta_0)\to R^\ell$ given by
\[
y_0(t):= P_1\circ \Phi(t,f_0),
\]
where $P_1 : D \to R^\ell$ denotes the projection in the first factor. The (components of the) function $y_0$ are semialgebraic and continuous curves on the right of the origin, thus the corresponding germs (which we denote with the same symbol) define algebraic Puiseux series $y_0 \in R\la\zeta\ra^\ell$. Furthermore, we note that by construction $(y_0(t),t) \in D$ for all $t\in (0,\delta_0)$ so that
\[\label{eq:fiber0}
\forall t\in (0,\delta_0), \qquad \#\left(S|_t\cap \pi^{-1}(y_0(t))\right) > N.
\]
Denoting by $(v,w)\in R\la \zeta\ra^{n} = R\la \zeta\ra^{n-\ell}\times R\la \zeta\ra^{\ell}$, the set $\pi^{-1}(y_0(t))\subset R^{n}$ can be seen as the evaluation at $t$ of the set $\pi^{-1}(y_0) \in R\la\zeta\ra^n$, defined by the formula $(w=y_0)$ with variables $(v,w)$ and coefficients in $R\la\zeta\ra$.

Thus for $0<t\ll 1$ (which means, according to \cref{def:suffsmall}, up to taking a smaller $\delta_0$, for all $0<t<\delta_0$), it holds
\begin{align}
\left.\left(S\cap \pi^{-1}(y_0)\right)\right|_t & = \mathrm{Reali}\Big((\psi\cap (w=y_0))|_{\zeta = t};R\Big) \\
& = \mathrm{Reali}\Big(\psi|_{\zeta =t};R\Big)\cap \mathrm{Reali}\Big((w=y_0)|_{\zeta =t};R\Big) \\
& = S|_t \cap \pi^{-1}(y_0(t)).
\end{align}

In particular from \eqref{eq:fiber0} we obtain
\[\label{eq:fiber}
 \forall t\in (0,\delta_0) \qquad \# \left.\left(S\cap \pi^{-1}(y_0)\right)\right|_t  = \#\left(S|_t\cap \pi^{-1}(y_0(t))\right) > N.
\]
Consider then the semialgebraic set $T\subset R^n\times R_+$ defined by
\[
T:=\left\{ (z,t)\in R^n\times R_+\,\Big|\, z\in\left.\left(S \cap \pi^{-1}(y_0)\right)\right|_{t}\right\}.
\]
Denote by $P: T\to R_+$ the projection on the last factor. By semialgebraic triviality, for $\delta_0>0$ small enough, we find a semialgebraic homeomorphism
\[
\varphi: (0,\delta_0)\times H \to P^{-1}((0,\delta_0)),
\]
where $H=P^{-1}(\bar{t})$ for some $\bar{t}\in (0,\delta_0)$. By \eqref{eq:fiber} such a fiber must have cardinality $>N$. Pick $N+1$ distinct points $h_1,\dots,h_{N+1}\in H$. Denote by $Q:T\to R^n$ the projection on the first factor and define for $j=1,\dots,N+1$ continuous semialgebraic curves $g_j:(0,\delta_0)\to R^n$ by
\[
g_j(t):=Q(\varphi(t,h_j)).
\]
By construction, these curves satisfy
\[
\forall t\in (0,\delta_0)\quad g_j(t)\in \left.\left(S\cap \pi^{-1}(y_0)\right)\right|_{t}.
\]
Therefore (the germ of) each $g_j \in S \cap\pi^{-1}(y_0)\subset R\la\zeta\ra^n$, by \eqref{property}. Since the elements $h_j$ are distinct, the germs of these curves near zero represent $N+1$ distinct points in $S\cap \pi^{-1}(y_0)$, contradicting \eqref{eq:supN}.

We have proved that there exists $\delta>0$ such that for all $t\in R$ with $0<t<\delta$ it holds
\[
\sup_{y\in R^\ell} \#S|_t\cap \pi^{-1}(y)\leq \sup_{w\in R\la\zeta\ra^\ell} \# S\cap \pi^{-1}(w).
\]
which is the statement.
\end{proof}

We will need a version of \cref{lemma:cardinalitaFormula-one} for multiple infinitesimals ($m>1$).
\begin{lemma}[Small evaluations do not increase the strong dimension: several infinitesimals]\label{lemma:cardinalitaFormula}
In the same setting of \cref{def:Lt}, let $\pi_\ell : R\la\zeta_1,\dots,\zeta_m\ra^{n}\to R\la\zeta_1,\dots,\zeta_m\ra^\ell$ be the projection on the last $\ell$ coordinates, $1\leq \ell\leq n$, and let $1\leq k\leq m$. Then for $0<t_k\ll\cdots \ll t_m\ll 1$ it holds
\[\label{eq:monoS}
\sup_{y \in R\la \zeta_1,\dots,\zeta_{k-1}\ra^\ell} \# S|_{t_k,\dots,t_m}\cap\pi_\ell^{-1}(y) \leq \sup_{w \in R\la \zeta_1,\dots,\zeta_{m}\ra^\ell} \# S\cap\pi_\ell^{-1}(w).
\]
In particular, if $S\subset R\la \zeta_1,\dots,\zeta_m\ra^n$ is strongly of dimension $\leq \ell$, then  for $0<t_k\ll\cdots \ll t_m\ll 1$ the set $S|_{t_k,\dots,t_m}$ is also strongly of dimension $\leq \ell$.
\end{lemma}
\begin{proof}
Note that \cref{lemma:cardinalitaFormula-one} corresponds to the case $m=1$. To prove the case $m>1$, recall that $R\la\zeta_1,\dots,\zeta_m\ra \simeq K\la\zeta\ra$, with $K:=R\la\zeta_1,\dots,\zeta_{m-1}\ra$. We find $\delta_m\in R\la\zeta_1,\dots,\zeta_{m-1}\ra$, $\delta_m>0$, such that for all $t_m \in R\la\zeta_1,\dots,\zeta_{m-1}\ra$ with $0<t_m<\delta_m$ it holds 
\[
\sup_{y \in R\la\zeta_1,\dots,\zeta_{m-1}\ra^\ell} \#S|_{t_m}\cap \pi^{-1}(y) \leq \sup_{w\in R\la\zeta_1,\dots,\zeta_{m}\ra^\ell} S \cap \pi^{-1}(w).
\]
We can now iterate the argument $k$-times, with $1\leq k\leq m$ to get the statement, using the definition of $S|_{t_k,\dots,t_m}$.
\end{proof}
}

 \begin{definition}[Bounded elements and the limit homomorphism: one infinitesimal]\label{def:boundedone}Let $R$ be a real closed field. An element $s\in R\la\zeta\ra$ is called \emph{bounded over $R$} if $\|s\|\leq r$ for some $r\in R_+$. The subring $R\la\zeta\ra_b$ of elements that are bounded over $R$ consists of algebraic Puiseux series with non-negative exponents.

For all $n\in \mathbb{N}$, we define the  \emph{limit homomorphism}
\[\label{eq:la}\lambda_{\zeta}:R\la\zeta\ra_b^n\to R^n,\]
by the ring homomorphism mapping $\sum_{k\geq0}^{\infty} a_k\zeta^{\frac{k}{m}}$ to $a_0\in R^n$.
\end{definition}

\begin{remark}
Viewing $R\la\zeta\ra$ as the field of germs of semialgebraic functions continuous on the right of zero, the bounded elements corresponds to those germs that have a finite limit as $\zeta\to 0$ and 
\[\lambda_\zeta(f)=\lim_{\zeta\to 0}f(\zeta).\]
From this we see that the map $\lambda_\zeta$ is order preserving, in the following sense: if $f_1,f_2\in R\la\zeta\ra_b$ with $f_1\leq f_2$, then $\lambda_\zeta(f_1)\leq \lambda_\zeta(f_2)$.
\end{remark}

 \begin{definition}[Bounded elements: several infinitesimals]\label{def:boundedseveral}Let $R$ be a real closed field. We say that $f=(f_1, \ldots, f_n)\in  R\la\zeta_1, \ldots, \zeta_m\ra^n$ is \emph{bounded  over $R$} if $\|f\|\leq r$ for some $r\in R_{+}$. 
 \end{definition}
 
\begin{remark}
The subring $R\la\zeta_1, \ldots, \zeta_m\ra^n_b$ of elements that are bounded over $R$ contains algebraic Puiseux series with non-negative exponents. However the inclusion is strict: the Puiseux series $t= \zeta_1^{-1}\zeta_2 \in R\la\zeta_1,\zeta_2\ra$ is such that $\|t\|\leq r$ for all $r\in R_+$, $r>0$.
\end{remark}

 Note that if $R$ is real closed, then on the set of elements of $R\la\zeta_1, \ldots, \zeta_m\ra^n$ that are bounded over $R$ (i.e. with $\|f\|\leq r$ with $r\in R_+$) the composition of maps
\[ R\la\zeta_1, \ldots, \zeta_m\ra_b^n\stackrel{\lambda_{\zeta_m}}{\longrightarrow}R\la\zeta_1, \ldots, \zeta_{m-1}\ra_b^n\stackrel{\lambda_{\zeta_{m-1}}}{\longrightarrow}\cdots R\la\zeta_1\ra_b^n\stackrel{\lambda_{\zeta_1}}{\longrightarrow}R^n\] is well--defined, since at every step we get bounded elements.
\begin{remark}
We stress that the composition above is well--defined only if taken with the order prescribed by the infinitesimals. For instance, let $f\in \R\la \zeta_1, \zeta_2\ra_{b}$ be given by
\[f=\sum_{k\geq 0}a_k(\zeta_1)\zeta_2^{\frac{k}{q}} \quad \textrm{where}\quad a_k(\zeta_1)=\sum_{j\geq 0}b_{k,j}\zeta_1^{\frac{j}{q_k}}.\]
(Note that each $a_k(\zeta_1)$ has ``its own'' $q_k$.) Then
\[f=\sum_{k\geq 0}\left(\sum_{j\geq 0}b_{k,j}\zeta_1^{\frac{j}{m_k}}\right)\zeta_2^{\frac{k}{m}},\]
and $\lambda_{\zeta_1}(\lambda_{\zeta_2}(f))=b_{0,0}$, whereas the composition in the other order is not well--defined.
\end{remark}

\begin{definition}[The map ${{\Lim}_0}$]\label{def:L0}
Let $R$ be a real closed field and let  $S\subset R\la\zeta_1, \ldots, \zeta_m\ra^n$ be a semialgebraic set bounded over $R$. We denote by
\[
\Lim_0 (S):=\lambda_{\zeta_1}\cdots\lambda_{\zeta_m}(S).
\]
\end{definition}

\begin{remark}
The set ${{\Lim}_0}(S)\subset R^n$ is closed. In fact, if $S$ is defined by a formula $\psi$ with coefficients in $R\la\zeta\ra$, then it follows from  \cite[Prop.\ 12.43]{BasuPoRoyBook} that
\begin{align}\label{eq:L0explicit}
	{{\Lim}_0}(S)
	&=\overline{\left\{(x, t)\in R^{n+1}\,\bigg|\, \left(x\in S|_t\right)\land(t>0)\right\}}\cap R^{n}  \\
&=\overline{\mathrm{Reali}\bigg(\psi|_{\zeta=t}\land(t>0);\R\bigg)}\cap R^{n},
\end{align}
where we identify $R^{n}$ with $\{u=0\}$. A similar argument holds for $m>1$. 
Notice, however, that deducing a presentation for $\Lim_0(S)$ as in \cref{def:semialgebraic} is more complicated and requires quantifier elimination, which would bring us back to the problem mentioned in \cref{remark:exponential}.
\end{remark}

\subsubsection{Hausdorff limits}
Going back to \cref{example:neighbourhood}, assuming that $A\subset \R^n$ is bounded, then also the corresponding $S\subset \R\la \zeta\ra^n$ is bounded. Moreover, we notice that $\Lim_0(S)=A.$ Then, the set $\mathcal{U}_r(A)= S|_r$ converges in the Hausdorff metric to $A={{\Lim}_0}(S) $.
  To state the analogue result in general, we need to recall some more preliminary notions.
  
 Let $R$ be a real closed field. Given a semialgebraic set $S\subset R^n$, the distance from $S$ is the function defined by
\[\delta_S(x):=\mathrm{inf}_{s\in S}\|x-s\|, \quad x\in R^n.\]
This is a continuous, semialgebraic function, vanishing on the closure of $S$ and positive elsewhere, see \cite[Prop.\ 2.2.8]{BCR}.
Given $S\subset R^n$ and $r\in R$, the $r$--neighbourhood of $S$ in $R^n$ is the set defined by
\[\mathcal{U}_r(S, R^n):=\left\{x\in R^n\,\bigg|\, \delta_S(x)\leq r\right\}.\]
Since $\mathrm{dist}_S(\cdot)$ is semialgebraic, for every $r>0$ the set $\mathcal{U}_r(S, R^n)$ is also semialgebraic. 
\begin{definition}[Semialgebraic Hausdorff distance]\label{def:hausdorff}The \emph{Hausdorff distance} between two semialgebraic sets $S_1, S_2\subset R^n$ is defined as
\[\mathrm{dist}_{H}(S_1, S_2):=\inf\left\{\epsilon \in R\,\bigg|\, S_1\subseteq \mathcal{U}_\epsilon(S_2),\, S_2\subseteq \mathcal{U}_\epsilon(S_2)\right\}.\]
\end{definition}
Note that, if $R=\R$, this gives the usual Hausdorff distance. However, for general real closed fields $R$, this is not a ``distance'' in the sense of metric geometry, since the values of this function are elements of $R$. Still, given three closed semialgebraic sets $S_1, S_2, S_3,$ we have 
\[\label{eq:thh}\mathrm{dist}_H(S_1, S_3)\leq \mathrm{dist}_H(S_1, S_2)+\mathrm{dist}_H(S_2, S_3).\] This is proved exactly as in the classical case $R=\R$.

The next result outlines a useful property related to the map \eqref{eq:la}.

\begin{proposition}Let $S_1, S_2\subset R\la\zeta\ra_b^n$ be semialgebraic sets. Then
\[\label{eq:dHl}\mathrm{dist}_{H}(\lambda_{\zeta}(S_1), \lambda_{\zeta}(S_2))\leq \lambda_\zeta\left(\mathrm{dist}_H(S_1, S_2)\right).\]
\end{proposition}
\begin{proof}We first prove the following fact. If $x, y\in R\la\zeta\ra_b^n$, then
\[\label{eq:norml}\|\lambda_{\zeta}(x)-\lambda_\zeta(y)\|= \lambda_\zeta\left(\|x-y\|\right).\]
In fact, since $\lambda_\zeta: R\la\zeta\ra_b^n\to R^n$ is a ring homomorphism, writing $u=\sum_{k\geq 0}a_k\zeta^{\frac{k}{q}}$, with $a_k\in R^n$, it holds\[\|\lambda_\zeta(u)\|= \|a_0\|=\lambda_\zeta(\|u\|).\] 

Let us now prove \eqref{eq:dHl}. Let $\epsilon \in R\la \zeta\ra$ such that $\mathrm{dist}_H(S_1, S_2)\leq \epsilon$. This is equivalent to the following statement: for every $a_1\in S_1$ there exists $a_2(a_1)\in S_2$ such that $\|a_1-a_2(a_1)\|\leq \epsilon$ and for every $b_2\in S_2$ there exists $b_1(b_2)\in S_1$ such that $\|b_1(b_2)-b_2\|\leq \epsilon.$

Then, for every $r_1=\lambda_\zeta(a_1)\in \lambda_\zeta(S_1)$  the element $r_2:=\lambda_\zeta(a_2(a_1))\in \lambda_\zeta(S_2)$ is such that, using \eqref{eq:norml},
\[\|r_1-r_2\|=\|\lambda_\zeta(a_1)-\lambda_\zeta(a_2(a_1))\|=\lambda_\zeta(\|a_1-a_2\|)\leq \lambda_\zeta(\epsilon).\]
Similarly, for every $s_2\in \lambda_\zeta(S_2)$ there exists and element $s_1\in \lambda_\zeta(S_1)$ such that
\[\|s_1-s_2\|\leq \lambda_\zeta(\epsilon).\]
This means that $\mathrm{dist}_H(\lambda_\zeta(S_1), \lambda_\zeta(S_2))\leq \lambda_\zeta(\epsilon)$.
 \end{proof}

Recall now the following result from  \cite{BasuLerario}.
\begin{proposition}[{\cite[Prop.\ 2.7]{BasuLerario}}]\label{propo:BL}
Let $S=\mathrm{Reali}(\psi; \R\la\zeta\ra)\subset \R\la\zeta\ra^n$ be a bounded semialgebraic set. Then, in the usual Hausdorff metric,
\[
\lim_{r\to 0} S|_r={{\Lim}_0}(S).
\] 
\end{proposition}
\begin{remark} The previous result is stated in \cite[Prop.\ 2.7]{BasuLerario} under the assumption that the formula $\psi$ has coefficients in $\R[\zeta]$. In fact the proof in the general case goes exactly in the same way, provided that $r$ is sufficiently small so that the evaluation is well--defined as explained in \cref{def:Lt}.
\end{remark}

\begin{remark}
Even if $S|_0$ is well--defined, in general, this may be far, in the Hausdorff metric, from $\Lim_0(S)$. For example let $\alpha=\left(x^2-\zeta^2<0\right)$  and $A=\mathrm{Reali}(\alpha; R\la\zeta\ra)$. Then $A|_0=\emptyset$ and $\Lim_0(A)=\{0\},$ so that $A|_0\subsetneq \Lim_0(A)$. On the other hand, if $\beta=(\zeta x=0)$, and $B=\mathrm{Reali}(\beta; R\la\zeta\ra)$,  then $B|_0=R$ and $\Lim_0(B)=\{0\}$, so that $B|_0\supsetneq \Lim_0(B).$
\end{remark}
We extend \cref{propo:BL} to any real closed field using the transfer principle.

\begin{proposition}\label{propo:baseind}
Let $F$ be a real closed extension of $\R$. Let $S=\mathrm{Reali}(\psi; F\la\zeta\ra)\subset F\la\zeta\ra^n$ be a bounded semialgebraic set. Then,
\[ 
\forall \epsilon>0 \quad \exists \delta>0\quad \forall 0< t<\delta\quad \mathrm{dist}_H(\mathrm{Reali}(\psi|_{\zeta = t}; F), \lambda_\zeta(S))<\epsilon.
\]
(Here all the variables $\epsilon, \delta,t$ are in $F$, and $\mathrm{dist}_H$ denotes the semialgebraic Hausdorff distance.)
\end{proposition}
\begin{proof}Given a semialgebraic set $S$ as in the statement of \cref{propo:BL}, we can rephrase the content of its conclusion by saying that, 
\[ 
\forall \epsilon>0 \quad \exists \delta>0\quad \forall 0<t<\delta\quad \mathrm{dist}_H(\mathrm{Reali}(\psi|_{\zeta=t};\R), \lambda_\zeta(S))<\epsilon.
\]
All variables in this formula are in $\R$ and $\mathrm{dist}_H$ is the usual Hausdorff distance, which is written as a first order formula with coefficients in $\R$. Therefore, the same conclusion holds if $\R$ is replaced by a real closed extension $F$ of it, by  \cite[Prop.\ 5.2.3]{BCR}.
\end{proof}

We now generalize \cref{propo:baseind} to the case of multiple infinitesimals.
\begin{theorem}\label{propo:zeta}
Let $R$ be a real closed extension of $\R$. Let $S=\mathrm{Reali}(\psi;R\langle\zeta_1, \ldots, \zeta_m\rangle)\subset R\langle\zeta_1, \ldots, \zeta_m\rangle^n$ be a semialgebraic set, bounded over $R$.  Then, for every $1\leq k\leq m$ it holds\begin{equation}
	\label{eq:limitorder}\textrm{``}\lim_{t_k\to 0}\cdots\lim_{t_m\to 0}S|_{t_k, \ldots, t_m}=\lambda_{\zeta_k}\cdots \lambda_{\zeta_m}(S)\textrm{''},
	\end{equation}
	where \eqref{eq:limitorder} means: 
	for all $\epsilon\in R\la \zeta_1, \ldots, \zeta_{k-1}\rangle$, $\epsilon>0$, for $0<t_k\ll\cdots \ll t_m\ll 1$ it holds
	\[
	\mathrm{dist}_H\Big(S|_{t_k, \ldots, t_m}, \lambda_{\zeta_k}\cdots \lambda_{\zeta_m}(S)\Big)< \epsilon,
	\]
			where $\mathrm{dist}_H$ is the semialgebraic Hausdorff distance (see \cref{def:hausdorff}).
	\end{theorem}
	\begin{proof}When $k=m$, the statement is given by \cref{propo:baseind} applied to the case $F=R\la \zeta_1, \ldots, \zeta_{m-1}\ra$, so that $F\la\zeta\ra\simeq R\la \zeta_1, \ldots, \zeta_{m-1}, \zeta\ra$, where $\zeta=\zeta_{m}$.
	
		Assume that the statement is true for $k=j\leq m$. We prove it for $k=j-1.$
		By our working assumption, for every $\epsilon\in R\la \zeta_1, \ldots, \zeta_{j-1}\ra$, $\epsilon>0$, for  $0<t_j\ll t_{j+1}\ll \dots \ll t_m \ll 1$ the following property holds:
	\[\label{eq:dh2} 
	\mathrm{dist}_{H}\left(S|_{t_j, \ldots, t_m}, \lambda_{\zeta_j}\cdots \lambda_{\zeta_m}(S)\right)< \frac{\epsilon}{2}.
	\]
Recall that here $t_i \in R\la\zeta_1, \ldots, \zeta_{i-1}\ra$ for every $j\leq i\leq  m$. A fortiori we can take $\epsilon\in R\la \zeta_1, \ldots, \zeta_{j-2}\ra$.

	Denote by $A:=S|_{t_j, \ldots, t_m}\subset R\la\zeta_1, \ldots, \zeta_{j-1}\ra^n$ and by $B:=\lambda_{\zeta_j}\cdots \lambda_{\zeta_m}(S)\subset R\la\zeta_1, \ldots, \zeta_{j-1}\ra^n$, which are both bounded over $R$. Given $\epsilon\in R\la \zeta_1, \ldots, \zeta_{j-2}\ra$, $\epsilon>0$ let $\delta_{j-1}\in R\la \zeta_1, \ldots, \zeta_{j-2}\ra$, $\delta_{j-1}>0$ be given by \cref{propo:baseind} such that for all $t_{j-1} \in R\la \zeta_1,\dots,\zeta_{j-2}\ra$ with $0<t_{j-1}<\delta_{j-1}$ the following property holds:
	\[\label{eq:S1}
	\mathrm{dist}_{H}\left(A|_{t_{j-1}}, \lambda_{\zeta_{j-1}}(A)\right)< \frac{\epsilon}{2}.
	\]
Then, for $0<t_{j-1}\ll t_j \ll \dots \ll t_m \ll 1$ it holds
	\begin{align}
	\mathrm{dist}_H(S|_{t_{j-1}, \ldots, t_m}, \lambda_{\zeta_{j-1}}\cdots \lambda_{\zeta_m}(S))&\stackrel{\eqref{eq:thh}}{\leq} \mathrm{dist}_H(A|_{t_{j-1}}, \lambda_{\zeta_{j-1}}(A))+\mathrm{dist}_H(\lambda_{\zeta_{j-1}}(A), \lambda_{\zeta_{j-1}}(B)),\\
	&\stackrel{\eqref{eq:S1}}{<} \frac{\epsilon}{2}+\lambda_{\zeta_{j-1}}\left(\mathrm{dist}_H(A,B)\right),\\
	&\stackrel{\eqref{eq:dHl}}{\leq}\frac{\epsilon}{2}+\lambda_{\zeta_{j-1}}\left(\frac{\epsilon}{2}\right)= \epsilon,
	\end{align}
	where in the last inequality we used the fact that $\epsilon$ was chosen in $R\la\zeta_1, \ldots, \zeta_{j-2}\ra$.
		\end{proof}

	\subsection{Hausdorff approximations of closed and bounded sets}\label{sec:hacb}
	Recall that every closed (in the Euclidean topology) semialgebraic set $S\subseteq R^n$ can be written as a finite union of closed basic semialgebraic sets (\cite[Thm.\ 2.7.2]{BCR}):
 \begin{equation}\label{Semialgebrici_Def}
S=\bigcup_{i=1}^{a}\bigcap_{j=1}^{b_i}\{x\in R^n \mid p_{ij}(x)\leq 0\}.
\end{equation}
In general, for a closed semialgebraic set $S$ with diagram $D(S)=(n, c, d)$, passing from a representation of the form \eqref{eq:defA} to one of the form \eqref{Semialgebrici_Def}, we cannot control the number of unions and intersections in \eqref{Semialgebrici_Def} as a function of $c$, nor the degrees of the polynomials in \eqref{Semialgebrici_Def} as a function of $d$. (On the other hand, in passing from \eqref{Semialgebrici_Def} to \eqref{eq:defA} the process is controlled.) 

However, here we use the tools from the previous section to show that, given a closed and bounded semialgebraic set with a representation as in \eqref{eq:defA}, we can \emph{approximate} it with a closed and bounded semialgebraic set described as in \eqref{Semialgebrici_Def} in a controlled way (\cref{propo:approxclosed}).

We start with the following elementary lemma.
\begin{lemma}\label{closurebasicset}
Let
$C \subset \R^n $ be of the form
\begin{equation}
C=\left(\bigcap_{j\in J^=}\left\{x\in \R^n\,\bigg|\,p_{j}(x)=0\right\}\right)\cap 
	\left( \bigcap_{j\in J^{<}}\left\{ x\in \R^n\,\bigg|\,q_{j}(x)<0\right\}\right),
\end{equation}
where the $p_j$ and the $q_j$ are polynomials and $J^=, J^<$ are finite sets. 
Then, denoting by $(x, u)$ points in $\R^n \times \R$, and identifying $\{u=0\}$ with $\R^n$, the closure of $C$ can be described as:
\begin{equation}\label{eq:closure}
\overline{C} 
=
\overline{\left\{(x, u) \,
	\bigg| \bigg(
	p_j (x) = 0,\, \forall j\in J^=\bigg)\land 
	\bigg(q_j (x) + u \leq 0, \, \forall j\in J^<\bigg)\land\bigg( u>0\bigg) \right \} }
\cap 
\{u=0\}.
\end{equation}
\end{lemma}
\begin{proof}
	We prove the two inclusions separately.
	To prove the inclusion of the set on the left of \eqref{eq:closure} into the one on the right, by the monotonicity of the closure operation, it is enough to show that 
	for any $z \in C$ there exists a sequence
	$\{ (z_k,u_k)\}_{k\geq1} $ converging to $ (z,0)$ satisfying for every $k\geq1$
	\begin{equation}
	\bigg(
	p_j (z_k) = 0,\, \forall j\in J^=\bigg)\land 
	\bigg(q_j (z_k) + u_k \leq 0, \, 
	\forall j\in J^<\bigg)\land\bigg( u_k>0\bigg).
	\end{equation}
	Since $z \in C$, we have
		\begin{equation}
		\bigg(
		p_j (z) = 0,\, \forall j\in J^=\bigg)
		\land 
		\bigg(q_j (z) + U \leq 0, \, 
		\forall j\in J^<\bigg),
	\end{equation}
	with 
	$U:= - \max \{q_j (z) \,|\, j\in J^< \} > 0$.
	Hence, setting 
	$u_k := \frac{U}{k} >0$ and $z_k := z$,
	we observe that the sequence $(z_k , u_k)$ satisfies all the above conditions and converges to $(z,0)$. 
	
	To prove the other inclusion, we take
	\begin{equation}
		( z , 0) \in
		\overline{\left\{(x, u)\in \R^n \times \R \,
			\bigg| \bigg(
			p_j (x) = 0,\, \forall j\in J^=\bigg)\land 
			\bigg(q_j (x) + u \leq 0, \, \forall j\in J^<\bigg)\land\bigg( u>0\bigg) \right \} }.
	\end{equation}
Hence, there exists a sequence $(z_k, u_k)$ converging to $(z,0)$ such that for any $k$
\begin{equation}
	\bigg(
	p_j (z_k) = 0,\, \forall j\in J^=\bigg)\land 
	\bigg(q_j (z_k) + u_k \leq 0, \, 
	\forall j\in J^<\bigg)\land\bigg( u_k>0\bigg).
\end{equation}
In particular $z_k \in C$ for all $k$, hence $z \in \overline{C}$.
\end{proof}

\begin{proposition}\label{propo:approxclosed}For every $\epsilon>0$ and for every closed and bounded semialgebraic set $S\subset \R^n$ with diagram $D(S)=(n, c, d)$, there exists a closed semialgebraic set $S'\subset \R^n$ satisfying
	\[\label{eq:e/2}\mathrm{dist}_{H}(S, S')\leq \epsilon\]
and such that
\[\label{eq:Seps}{S'}
=
\bigcup_{i=1}^a 
\mathrm{Bas}(\mathcal{P}_{i}, \mathcal{Q}_i;\R),\]
with the property that, for every $i=1, \ldots, a$ we have $\#(\mathcal{P}_i\cup \mathcal{Q}_i)\leq b$, with $ab\leq c$, and with each polynomial in $\mathcal{P}_i\cup \mathcal{Q}_i$ of degree bounded by $d$.
\end{proposition}
\begin{proof}Let us write
	\begin{equation}S=\bigcup_{i=1}^{a} \bigcap_{j=1}^{b_i}\left\{x\in \R^n\,\bigg|\,\mathrm{sign}(p_{ij}(x))=\sigma_{ij}\right\},
	\end{equation}
	with $c=a\max_i\{b_i\}$.
	Let us examine the sets
	\[C_i:=\bigcap_{j=1}^{b_i}\left\{x\in \R^n\,\bigg|\,\mathrm{sign}(p_{ij}(x))=\sigma_{ij}\right\}.\]
	By possibly relabelling the $p_{ij}$ and multiplying them by $\pm1$, we can write each $C_i$ as
	\[C_i=\left(\bigcap_{j\in J^=_{i}}\left\{x\in \R^n\,\bigg|\,p_{ij}(x)=0\right\}\right)\cap 
	\left( \bigcap_{j\in J_{i}^{<}}\left\{ x\in \R^n\,\bigg|\,p_{ij}(x)<0\right\}\right),\]
	where $\#\left(J_i^=\cup J_i^<\right)= b_i.$ For each $i=1, \ldots, a$ and for every $i\in J_{i}^<$ we define the polynomial $\widetilde{p}_{ij}\in \R\la\zeta\ra[x_1, \ldots, x_n]$ as follows:
	\[\widetilde{p}_{ij}(x):=p_{ij}(x)+\zeta.
	\]
	For any $i= 1, \dots, a$ we consider the semialgebraic set 
	$\widetilde{S}_i \subset \R\la\zeta\ra^n$ 
	defined by the formula:
	\[\widetilde{\psi}_i
	:=
	\left(\bigwedge_{j\in J_i^=}p_{ij}(x)=0\right)\land\left(\bigwedge_{j\in J_{i}^<}\widetilde{p}_{ij}(x)\leq0\right).
	\]
	
	We claim that, if $S\subseteq B_{\R^n}(\rho)$ for $\rho>0$, then $\widetilde{S}_i \subseteq B_{\R\la\zeta\ra^n}(\rho)$. In fact, let  $g:(0, \delta)\to \R$ be a representative for an element of $\widetilde{S}_i$. Then, for every $\zeta\in (0, \delta)$ we have $p_{ij}(g(\zeta))=0$ for $j\in J_i^=$ and $p_{ij}(g(\zeta))+\zeta\leq 0$ for $j\in J_i^{<}$. Since $\zeta<\delta$, the inequality $p_{ij}(g(\zeta))+\zeta\leq 0$ implies $p_{ij}(g(\zeta))<0$,  i.e. $g(\zeta)\in C_i\subseteq B_{\R^n}(\rho)$ for every $\zeta\in (0, \delta)$. Therefore $g\in B_{\R\la\zeta\ra^n}(\rho)$.
	
	In particular $\widetilde{S}_i$ is bounded and, by \cref{propo:BL}, for any $\epsilon >0 $ we get 
	$r >0$ such that
	\begin{equation}\label{proof:approxanddiagram1}
		\mathrm{dist}_{H}
		\left( { {\Lim}_0}(\widetilde{S}_i) , \widetilde{S}_i|_{r})
			\right)
		\leq \epsilon.
	\end{equation}
	We have by \eqref{eq:L0explicit} and \cref{closurebasicset}
	\begin{equation}
		{\Lim}_0 (\widetilde{S}_i)
		 = \overline{
			\left\{(x, u)\in \R^{n+1}\,
			\bigg|\, 
			\left(x\in \widetilde{S}_i|_{u}
			\right)
			\land(u>0) \right\}} \cap \R^{n}   
		 =  \overline{C_i}.
	\end{equation}
	Observe that $ \widetilde{S}_i|_{r} = \mathrm{Bas}(\mathcal{P}_{i}, \mathcal{Q}_i;\R)$, where
\begin{equation}
\mathcal{P}_i=\left\{p_{ij}\,\big|\, j\in {J_i^=}\right\}\quad \textrm{and}\quad \mathcal{Q}_i
=
\left\{
p_{ij}+r \,\big|\, j\in J_{i}^<
\right \}.
\end{equation}
Note that  $\#(\mathcal{P}_i\cup \mathcal{Q}_i)\leq b:=\max_i b_i$ so that $ab\leq c$, and each polynomial in $\mathcal{P}_i\cup \mathcal{Q}_i$ has degree bounded by $d$. We now define
\[S' : = \bigcup_{i=1}^a  C_i',\qquad \text{with} \qquad C_i':=\mathrm{Bas}(\mathcal{P}_{i}, \mathcal{Q}_i;\R).
\]

Since $S$ is closed, we have $S=\bigcup_{i=1}^a \overline{C_i}$. From \eqref{proof:approxanddiagram1} we get directly
\begin{equation}
\mathrm{dist}_{H}
( S , S')\leq \max_i \mathrm{dist}_H\left(\overline{C_i},C_i'\right)
\leq \epsilon,
\end{equation}
concluding the proof. 
\end{proof}

\section{Quantitative approximate definable choices}\label{sec:quantitativeselection}
\subsection{Preliminary constructions}\label{sec:construction}

Following a construction from \cite{BR}, given a closed basic semialgebraic set 
\[ 
S=\B(\mathcal{P}, \mathcal{Q};\R)\subset\R^{n}
\] 
we construct a closed basic semialgebraic set $\widetilde{S} \subset R^{n}$, where $R$ is a field of algebraic Puiseux series with coefficients in $\R$, such that $\Lim_0(\widetilde{S})=S$ and 
for which a definable choice over a given projection can be made quantitatively.

\begin{construction}\label{remark:construction}
The following construction is taken from \cite{BR}. Set
\begin{equation}
	R= \R\langle \zeta_1, \zeta_2, \zeta_3\rangle.
\end{equation}
For a given closed basic and bounded semialgebraic set 
\[ S=\B(\mathcal{P}, \mathcal{Q};\R)\subset\R^{n},\]  and for every $1\leq k\leq n$ the construction provides new semialgebraic sets 
\[ 
\widetilde{S}^k\subseteq \widetilde{S}\subset R^{n}.
\] 

Assume that the polynomials in $\mathcal{P}, \mathcal{Q}$ have degree bounded by $d$.  First, using a perturbation argument, one  constructs new families of polynomials $\widetilde{\mathcal{P}}, \widetilde{\mathcal{Q}}\subset R[x_1, \ldots, x_{n}]$  with degrees bounded by $2d+2$. The cardinality of $\widetilde{\mathcal{Q}}$ equals the one of $\mathcal{Q}$, whereas (for our purposes) the cardinality of $\widetilde{\mathcal{P}}$ can be assumed to be $1$  (by taking the sum of the squares of its elements).  In this way we get the set 
	\begin{equation}
	\widetilde{S}:=\B(\widetilde{\mathcal{P}}, \widetilde{\mathcal{Q}}; R).
	\end{equation}
	Then, one constructs $\widetilde{S}^k$. First, set 
	\[g(x):=1+\sum_{j=1}^{n}j x_j^{2d+2},\] and 
	for every family $\mathcal{F}=\{f_1,\ldots, f_s\}\subset R[x_1, \ldots, x_{n}]$   define the algebraic set:
	\begin{equation}\label{eq:critf}
	\mathrm{Crit}_k(\mathcal{F}; R):=\{x\in R^{n}\mid f_1(x)=\cdots=f_s(x)=0,\, \mathrm{rank}(\widetilde{J}(x))\leq k\},
	\end{equation}
	where $\widetilde{J}(x)$ is the matrix of size $(n-k)\times (s+1)$ whose columns are the partial derivatives of $f_1, \ldots, f_s, g$ with respect to $x_{k+1}, \ldots, x_{n}$. 
		The set $\widetilde{S}^k$ is then defined by
		\[\label{eq:Sl}\widetilde{S}^k:=\bigcup_{\widetilde{\mathcal{Q}}'\subset \widetilde{\mathcal{Q}}}\mathrm{Crit}_k(\widetilde{\mathcal{P}}\cup \widetilde{\mathcal{Q}}'; R)\cap \widetilde{S},\]
		where the union runs over all subsets $\widetilde{\mathcal{Q}}'\subset \widetilde{\mathcal{Q}}  \subset R[x_1,\dots,x_n]$.
	\end{construction}
	\begin{remark}[On boundedness]Even though this is not explicitly stated in \cite{BR} (but indeed used in their arguments), by inspecting the form of $\widetilde{\mathcal{P}}$ and $\widetilde{\mathcal{Q}}$ provided there, one can see that, if $S\subset \R^{n}$ is bounded, then  $\widetilde{S}^\ell\subseteq \widetilde{S}\subset R^{n}$ are bounded over $\R$.
	\end{remark}
	
	The next result gives a control on the diagram of $\widetilde{S}^k$.
	
\begin{lemma}\label{lemma:diagram}
For every $ a_1\in \mathbb{N}$ there exist $a_2, a_3\in \mathbb{N}$ such that, for a closed basic semialgebraic set $S=\B(\mathcal{P}, \mathcal{Q};\R)\subset\R^{n}$, with $\#\mathcal{P}, \#\mathcal{Q}\leq a_1$ and with every polynomial in $\mathcal{P}\cup \mathcal{Q}$ of degree bounded by $d$, and any $1\leq k  \leq n$, it holds
	\begin{equation}
	D(\widetilde{S}^{k})=(n, a_2, a_3 d),
	\end{equation}
	where $\widetilde{S}^{k} \subset R^{n}$ is the set associated with $S$ by \cref{remark:construction}. 
\end{lemma}
\begin{proof}Let $\widetilde{\mathcal{P}}, \widetilde{\mathcal{Q}}\subset R[x_1, \ldots, x_{n}]$ be the family of polynomials of \cref{remark:construction} such that 	$\widetilde{S}=\B(\widetilde{\mathcal{P}}, \widetilde{\mathcal{Q}};R)$. These polynomials have degree bounded by $2d+2$.
	
	Then, for every subset $\widetilde{\mathcal{Q}}'\subset \widetilde{\mathcal{Q}}$ (there are only finitely many such subsets, say at most $a_4=a_4(a_1)$) one considers the family $\mathcal{F}=\widetilde{\mathcal{P}}\cup \widetilde{\mathcal{Q}}'=\{f_1, \ldots, f_s\}$ (here $s\leq a_4+1$) and defines the algebraic set $\mathrm{Crit}_{k}(\mathcal{F})$ as in \eqref{eq:critf}. Since $R$ is real closed, this algebraic set is defined by a single polynomial equation $P_{\widehat{\mathcal{Q}}'}(x)=0$,
	where 
	\begin{equation}
	P_{\widehat{\mathcal{Q}}'}=\sum_{i=1}^s f_i^2+\sum m_{ij}^2.
	\end{equation}
	Here, the second sum runs over all the $(k+1)\times (k+1)$ minors of $\widetilde{J}(x)$. Remember that, $\widetilde{J}(x)$ is a $(n-k)\times (s+1)$ matrix, hence if $k> s$ there are no such minors, so we can assume that $k\leq s$. Therefore the polynomial $P_{\widehat{\mathcal{Q}}'}$ has degree
	\begin{equation}
	\deg\left(P_{\widehat{\mathcal{Q}}'}\right)\leq  (k+1)(4d+4)  \leq (s+1)(4d+4) \leq a_5 d,
	\end{equation}
 where $a_5=a_5(a_1)$. 
 
 The set $\widetilde{S}^\ell$ from \cref{remark:construction} can therefore be written as:
	\begin{align}
	\widetilde{S}^\ell & = \bigcup_{\widetilde{\mathcal{Q}}'\subset \widetilde{\mathcal{Q}}}\mathrm{Crit}_{k}(\widetilde{\mathcal{P}}\cup \widetilde{\mathcal{Q}}';R)\cap \widetilde{S}=\widetilde{S}\cap\bigcup_{\widehat{\mathcal{Q}}'\subset \widehat{\mathcal{Q}}}Z(P_{\widehat{\mathcal{Q}}'};R)=\widetilde{S}\cap Z(\widetilde{F};R) \\
	&  = Z(\widetilde{\mathcal{P}};R)  \bigcap_{\tilde{q}\in\widetilde{\mathcal{Q}}} \{\tilde{q} \leq 0\} \cap Z(\widetilde{
	F};R),\label{eq:firstpresStildel}
	\end{align}
	where $\widetilde{\mathcal{P}},\widetilde{\mathcal{Q}}$ are the families defining the basic set $\widetilde{S}$, and
	\begin{equation}
	\widetilde{F}:=\prod_{\widetilde{\mathcal{Q}}'\subset \widetilde{\mathcal{Q}}}P_{\widetilde{\mathcal{Q}}'},
	\end{equation}
	which is a polynomial of degree bounded by $a_4a_5 d=a_3 d$, with $a_3$ depending only on $a_1$. To obtain a presentation as in \eqref{eq:defA} note that if $\{p_1,\dots,p_L\}$ is a family of polynomials, then
	\begin{equation}\label{eq:rewriting}
			\bigcap_{i =1}^L \{ p_i \leq 0\}  = 	\bigcap_{i = 1}^L \Big( \{ p_i < 0\} \cup \{p_i = 0\} \Big)
		= \bigcup_{\sigma} \bigcap_{i=1}^L \{\mathrm{sign}(p_i) = \sigma_i\},
		\end{equation}
	where $\sigma$ runs over all possible choices in $\{0,-1\}^{L}$. We can apply identity \eqref{eq:rewriting} to \eqref{eq:firstpresStildel}. Since $\#\widetilde{\mathcal{P}}, \#\widetilde{\mathcal{Q}}\leq a_1$, we get that $\widetilde{S}^\ell$ admits a representation as in \eqref{eq:defA} with $a=2^L$, $b_i = L$, for $L=2a_1+1$. In other words $D(\widetilde{S}^\ell)=(n,a_2,a_3 d)$, for $a_2 = (2a_1+1)2^{2a_1+1}$.
\end{proof}
The sets from  \cref{remark:construction} enjoy the following properties. For $1\leq k\leq n$ let \[
\pi_k:R^{n}\to R^k
\] 
be the projection onto the last $k$ coordinates. Since $\R \subset R  =\R\la\zeta_1,\zeta_2,\zeta_3\ra$, the same symbol $\pi_k:\R^{n}\to \R^k$ is used to denote the restriction to $\R^n$, without risk of confusion.

\begin{proposition}\label{prop:5properties}
Let $S=\mathrm{Bas}(\mathcal{P}, \mathcal{Q};\R)\subset \R^{n}$ be a closed basic and bounded semialgebraic set such that, for some $1\leq k\leq n$ it holds
\begin{equation}\label{eq:finite}
\forall y\in \R^k\quad \#\left(Z(\mathcal{\mathcal{P}};\R)\cap \pi_{k}^{-1}(y)\right)<\infty,
\end{equation}
i.e.\ $Z(\mathcal{\mathcal{P}};\R)$ is strongly of dimension $\leq k$. Then for any $\ell < k$, the sets $\widetilde{S}^\ell\subseteq \widetilde{S}\subset \R\la\zeta_1,\zeta_2,\zeta_3\ra^n$ from \cref{remark:construction} are closed and bounded over $\R$. Moreover they satisfy the following properties.  For every $w\in \R\la\zeta_1,\zeta_2,\zeta_3\ra^\ell$
	\begin{enumerate}[(i)]
		\item \label{property1} the set $\widetilde{S}^\ell\cap \pi_\ell^{-1}(w)$ is finite, i.e.\ $\widetilde{S}^\ell$ is strongly of dimension $\leq \ell$;
		\item \label{property2} if $\widetilde{S}\cap\pi_{\ell}
		^{-1}(w)\neq \emptyset$, the set $\widetilde{S}^\ell\cap \pi_\ell^{-1}(w)$ intersects every connected component of $\widetilde{S}\cap\pi_\ell	^{-1}(w)$;
		\item \label{property3} it holds ${{\Lim}_0}(\widetilde{S}) = S$.
	\end{enumerate}
	Furthermore, for every $\epsilon\in \R$, $\epsilon>0$, for $0<t_1\ll \dots \ll t_m \ll 1$ the evaluation $S|_{t_1,\dots,t_m}$ is defined and the following properties hold:
\begin{enumerate}[label=(\roman*), start=4]
\item \label{property4} $\displaystyle  \mathrm{dist}_H(\widetilde{S}|_{t_1,\dots,t_m},S)<\epsilon$;
\item\label{property5} $\displaystyle  \mathrm{dist}_H\left(\pi_\ell(\widetilde{S}|_{t_1,\dots,t_m}),\pi_\ell(S)\right)<\epsilon$.
\end{enumerate}	
\end{proposition}
\begin{proof}
The fact that $\widetilde{S}^\ell\subseteq \widetilde{S}\subset\R\la\zeta_1,\zeta_2,\zeta_3\ra^n$ are closed, and bounded follows directly by their construction. \cref{property1,property2,property3} are proved in \cite[Prop.\ 5.5 and Prop.\ 5.17]{BR}, while \cref{property4} follows from \cref{property3} and \cref{propo:zeta}. To conclude, observe that $\pi_\ell : \R^n\to \R^\ell$ is $1$-Lipschitz with respect to $\mathrm{dist}_H$, therefore \cref{property5} follows from \cref{property4}.
\end{proof}

 \subsection{Approximate definable choice: the case of a projection}

The purpose of this section is to prove \cref{thm:appsel-intro}. We begin with the following preliminary version of the latter, for the case of closed basic sets.

 \begin{proposition}\label{propo:intermediate}
Let $\pi:\R^{n}\to \R^\ell$, with $1\leq \ell\leq n$ be the projection onto the last $\ell$ coordinates. For every $a_1\in \mathbb{N}$ there exist $a_2, a_3\in \N$ such that the following holds. For every closed basic and bounded semialgebraic set 
	\begin{equation}
	S=\B(\mathcal{P}, \mathcal{Q};\R) \subset \R^{n},
	\end{equation}
	with $\#\mathcal{P}, \#\mathcal{Q}\leq a_1$ and with every polynomial in $\mathcal{P}\cup \mathcal{Q}$ of degree bounded by $d$, and for all $\epsilon>0$ there exists a semialgebraic set $A_\epsilon \subset \R^{n}$ satisfying the following properties:
	\begin{enumerate}[(i)]
		\item \label{item:intermediate1} $\dim(A_\epsilon)\leq \ell$;
		\item \label{item:intermediate2}$A_\epsilon \subseteq \mathcal{U}_\epsilon(S)$;
		\item \label{item:intermediate3}$\mathrm{dist}_{H}(\pi(A_\epsilon), \pi(S))\leq \epsilon$;
		\item \label{item:intermediate4} $D(A_\epsilon)=(n, a_2, a_3 d)$.
	\end{enumerate}
\end{proposition}
\begin{proof}If  $\dim(S)\leq \ell$ and one can simply take $A_\epsilon=S$. 

Assume then that $k=\dim(S)>\ell$. Let $\widetilde{S}^\ell\subseteq \widetilde{S} \subset R^{n}$, where $R = \R\la\zeta_1,\zeta_2,\zeta_3\ra$ be the set defined by \cref{remark:construction}.  The strategy is to define
\begin{equation}\label{eq:Sepsdef}
A_\epsilon:=\widetilde{S}^\ell|_{t_1,t_2,t_3}
\end{equation}
for an appropriate choice of ``sufficiently small'' $(t_1,t_2,t_3)$ and to use \cref{prop:5properties}.
	However, in order to use it we need first to ensure that  the condition \eqref{eq:finite} is satisfied.  In order to do it, using \cref{propo:mather}, we can first perform a small linear change  of variables $L:\R^{n}\to \R^{n}$ (i.e.\ we choose $L$ sufficiently close to the identity) in the  defining polynomials and get new families $\mathcal{P}', \mathcal{Q}'$ with the property that: $\#\mathcal{P}'=\#\mathcal{P}$, $\#\mathcal{Q}'=\#\mathcal{Q}$; the degrees of all elements of $\mathcal{P}', \mathcal{Q}'$ are bounded by $d$; the Hausdorff distance between 
	$\mathrm{Bas}(\mathcal{P}, \mathcal{Q};\R)$ and $\mathrm{Bas}(\mathcal{P}', \mathcal{Q}';\R){ = L (\mathrm{Bas}(\mathcal{P}, \mathcal{Q};\R))}$ is arbitrarily small; and condition \eqref{eq:finite} is satisfied.
	
	It is elementary to verify that a set $A_\epsilon'\subset \R^{n}$ satisfying the properties of the statement for the semialgebraic set $\mathrm{Bas}(\mathcal{P}',\mathcal{Q}';\R)$ will also satisfy the same properties for the original set $\mathrm{Bas}(\mathcal{P}, \mathcal{Q};\R)$, up to adjusting $\epsilon$.
	Therefore, without loss of generality, we  assume that  \eqref{eq:finite} is  verified for the set $S$.

Set $A_\epsilon$ as in \eqref{eq:Sepsdef}. We show how to chose $(t_1,t_2,t_3)$ so that the desired properties hold.
 
	\begin{enumerate}[(i)]
		\item  By \cref{property1} of \cref{prop:5properties}, the set $\widetilde{S}^\ell$ is strongly of dimension $\leq \ell$. By \cref{lemma:cardinalitaFormula} for $0<t_1 \ll t_2 \ll t_3 \ll 1$ also $\widetilde{S}^\ell|_{t_1,t_2,t_3}$ is strongly of dimension $\leq \ell$.
			\item  By \cref{property4} of \cref{prop:5properties}, for $0<t_1 \ll t_2 \ll t_3\ll 1$ we have 
			\[
			\mathrm{dist}_H(\widetilde{S}|_{t_1,t_2,t_3},S)<\epsilon.
			\] 
			Since $\widetilde{S}^\ell\subseteq \widetilde{S}$ (in fact the formula defining $\widetilde{S}^\ell$ contains the formula defining $\widetilde{S}$ as a conjunction, by definition \eqref{eq:Sl}), for such $(t_1,t_2,t_3)$ we also have 
				 $\widetilde{S}^{\ell}|_{t_1,t_2,t_3} \subseteq \mathcal{U}_\epsilon(S)$.
				 
		\item By \cref{property5} of \cref{prop:5properties}, for $0<t_1 \ll t_2 \ll t_3\ll 1$ we have
		\[
		 \mathrm{dist}_H(\pi(\widetilde{S}|_{t_1,t_2,t_3}),\pi(S))<\epsilon.
		 \]  
		\item By \cref{lem:diagramext} and \cref{lemma:diagram}, for $0<t_1 \ll t_2 \ll t_3\ll 1$ we have 
		\[
		D(\widetilde{S}^\ell|_{t_1,t_2,t_3})=D(\widetilde{S}^\ell )=( n, a_2, a_3 d),
		\] 
		where $a_2,a_3 \in \N$ depend only on $a_1$.
	\end{enumerate}
	Since we are requiring only a finite number of properties, given $\epsilon>0$, the quantifier $0<t_1 \ll t_2 \ll t_3\ll 1$ can be commonly chosen  (see \cref{def:suffsmall}) so that the set \eqref{eq:Sepsdef} satisfies simultaneously \cref{item:intermediate1,item:intermediate2,item:intermediate3,item:intermediate4}. 
\end{proof}
Extending the previous result to general (non-basic) semialgebraic sets, we obtain the following statement, that corresponds to \cref{thm:appsel-intro}.

\begin{theorem}\label{thm:appsel} 
For every $c \in\mathbb{N}$ there exist $\kappa\in \N$ such that the following holds. Let $n,\ell,d\in \ell$, with $1\leq \ell\leq n$. Let $\pi:\R^{n}\to \R^\ell$ be the projection onto the last $\ell$ coordinates and let $S\subset \R^{n}$ be a bounded closed semialgebraic set with 
\[
D(S)=(n, c, d).
\] 
Then, for every $\epsilon>0$ there exists a closed semialgebraic set $	A_\epsilon\subset \R^{n}$ such that:
	\begin{enumerate}[(i)]
		\item \label{item:intermediate1t} $\dim(A_\epsilon)\leq \ell$;
		\item \label{item:intermediate2t}$A_\epsilon\subseteq \mathcal{U}_\epsilon(S)$;
		\item \label{item:intermediate3t}$\mathrm{dist}_{H}(\pi(A_\epsilon), \pi(S))\leq \epsilon$;
		\item \label{item:intermediate4t} $D(A_\epsilon)=(n, \kappa, \kappa d)$.
	\end{enumerate}
\end{theorem}
\begin{proof}
Applying \cref{propo:approxclosed} we find a closed semialgebraic set $S'\subset \R^{n}$ satisfying
	\[\label{eq:e/2p}
	\mathrm{dist}_{H}(S, S')\leq \frac{\epsilon}{2},
	\]
(in particular, $S'$ is also bounded) and such that
\[\label{eq:Sepsp}
S'=\bigcup_{i=1}^a \mathrm{Bas}(\mathcal{P}_{i}, \mathcal{Q}_i;\R),
\]
with the property that, for every $i=1, \ldots, a$ we have $\#(\mathcal{P}_i\cup \mathcal{Q}_i)\leq b$, with $ab\leq c$, and with each polynomial in $\mathcal{P}_i\cup \mathcal{Q}_i$ of degree bounded by $d$. For every $i=1, \ldots, a$, denote by $S_i':=\mathrm{Bas}(\mathcal{P}_{i}, \mathcal{Q}_i;\R)$ and apply  \cref{propo:intermediate} to each $S_i'$ to get sets $A_{i, \epsilon}$ satisfying the conclusions of \cref{propo:intermediate} (with $\varepsilon/2$ in place of $\varepsilon$). Define
\[\label{eq:Sepsdeft}
A_\epsilon:=\bigcup_{i=1}^a A_{i,\epsilon}.
\]
\begin{enumerate}[(i)]
\item Since for every $i=1, \ldots, a$ we have $\dim(A_{i, \epsilon})\leq \ell$, \cref{item:intermediate1t} follows.

\item As for \cref{item:intermediate2t}, this follows from the fact that $A_\epsilon\subseteq \mathcal{U}_\frac{\epsilon}{2}(S')$ and from \eqref{eq:e/2p}.

\item Similarly, for \cref{item:intermediate3t}, we have
\[
\mathrm{dist}_H(\pi(A_\epsilon), \pi(S))\leq \mathrm{dist}_H(\pi(A_\epsilon), \pi(S'))+\mathrm{dist}_H(\pi(S'), \pi(S))\leq \frac{\epsilon}{2}+\frac{\epsilon}{2} \leq \epsilon.
\]

\item By \cref{item:intermediate4} from \cref{propo:intermediate}, $D(A_{i, \epsilon})= (n, a_2(b), a_3(b)d)$. Therefore $D(A_\epsilon)= (n, \kappa', \kappa' d)$ for some $\kappa'=\kappa'(a_2(b), a_3(b), a)\leq \kappa(c)$. This proves \cref{item:intermediate4t}.
\end{enumerate}
The set $A_\epsilon$ defined by \eqref{eq:Sepsdeft} satisfies \cref{item:intermediate1t,item:intermediate2t,item:intermediate3t,item:intermediate4t}.
\end{proof}

\subsection{Approximate definable choice: the case of a semialgebraic map}

In this section we prove \cref{thm:SemialgebraicSelection-intro}, which we restate here.

\begin{theorem}\label{thm:SemialgebraicSelection}
	For every $ c, d, \ell \in \N$ there exists  $\beta>1$ satisfying the following statement. Let $n\in \N$ and let $K\subset \R^n$ be a closed semialgebraic set contained in the ball $B_{\R^n}(\rho)$, for some $\rho>0$, and let $F :\R^n\to \R^\ell$ be a locally Lipschitz semialgebraic map such that
\begin{equation}
D(\mathrm{graph}(F|_K)) = (n+\ell,c,d).
\end{equation}

Then for every $\epsilon \in (0 , \rho)$ there exists a closed semialgebraic set $C_\epsilon\subset \R^{n}$ such that:
	\begin{enumerate}[(i)]
		\item \label{item:thm1} $\dim(C_\epsilon)\leq \ell$;
		\item  \label{item:thm2} $C_\epsilon \subseteq \mathcal{U}_\epsilon(K)$;
		\item  \label{item:thm3} $\mathrm{dist}_{H}(F(C_\epsilon), F(K))\leq L(F, \rho) \cdot \epsilon$, where
		 $L(F, \rho):=2+\mathrm{Lip}(F, B_{\R^n}(2 \rho))$;
		\item  \label{item:thm4} for every $e=1, \ldots, n$ and every affine space $\R^e\simeq E\subseteq \R^n$, the number of connected components of $E\cap C_\epsilon$ is bounded by

		\begin{equation}b_0(E\cap C_{\epsilon})\leq \beta^e.\end{equation}
	\end{enumerate}
\end{theorem}
\begin{proof}
Let $\R^{n+\ell}=\R^n\times \R^\ell$ and denote by $\pi_1:\R^{n+\ell}\to \R^n$ and $\pi_2:\R^{n+\ell}\to \R^\ell$ the projections on the two factors. 

Given $\epsilon>0$, let $A_\epsilon\subset \R^{n+\ell}$ be the semialgebraic set obtained by applying \cref{thm:appsel} to $S = \mathrm{graph}(F|_K) \subset \R^{n+\ell}$, noting that by assumption the latter is a bounded and closed semialgebraic set with $D(S) = (n+\ell,c,d)$.
	 We define
	\begin{equation}
	C_\epsilon:=\pi_1(A_\epsilon).
	\end{equation}
	This set is semialgebraic and closed (it is a continuous image of a compact semialgebraic set).
	
	We verify that the set $C_\epsilon$ has the desired properties.
	\begin{enumerate}[(i)]
		\item  By \cref{item:intermediate1t} of \cref{thm:appsel}, $\dim(A_\epsilon)\leq \ell$. Therefore, by \eqref{eq:dimension2}, $\dim(C_\epsilon)\leq \ell$.
		\item Let $x\in C_\epsilon$. Then, by \cref{item:intermediate2t} of \cref{thm:appsel}, there exists $y$ such that $x=\pi_1(x, y)$ with $(x, y)\in A_\epsilon \subseteq \mathcal{U}_\epsilon(S)$. Therefore there exists $(x', y')\in S$ such that 
		\begin{equation}
		\|x-x'\|+\|y-y'\|\leq \epsilon.
		\end{equation}
		This implies that $\|x-x'\|\leq \epsilon$ and, since $x'=\pi_1(x', y')\in \pi_1(S)=K$, we get $x\in \mathcal{U}_\epsilon(K)$.
		Since this is true for all $x\in C_\epsilon$, then $C_\epsilon\subseteq \mathcal{U}_\epsilon(K)$.
		\item
		 Recall that we have set $L(F, \rho)=2 + \mathrm{Lip}(F, B_{\R^n}(2\rho))=:L$. We need to prove the two inclusions $F(C_\epsilon)\subseteq \mathcal{U}_{L\epsilon}(F(K))$ and $F(K)\subseteq \mathcal{U}_{L\epsilon}(F(C_\epsilon))$.

			We first prove that $F(C_\epsilon)\subseteq \mathcal{U}_{L\epsilon}(F(K))$.
		 Pick $c \in C_{\epsilon}$, then there exists $s \in A_{\epsilon}$ such  that $c= \pi_1 (s)$. 
		 Since $A_{\epsilon} \subseteq \mathcal{U}_{\epsilon} (S)$ (by \cref{item:intermediate2t} of \cref{thm:appsel}), there exists $z \in S$ such that 
		 $\norm{z - s} \leq \epsilon$. 
		 Now, $\pi_1(z) \in K$ and we have 
		 $\norm{c - \pi_1(z)} 
		 = 
		 \norm{\pi_1(s) - \pi_1(z)} \leq \norm{s - z} \leq \epsilon $. 
		 Hence, for $\epsilon \in (0,\rho)$ we have
		 \begin{align}
			\norm{F(c) - F(\pi_1(z))} 
			& \leq 
			\mathrm{Lip}(F, B_{\R^n}(2\rho))
			\norm{c - \pi_1(z)} \\
			& \leq \mathrm{Lip}(F, B_{\R^n}(2\rho)) \epsilon \leq L(F, \rho) \epsilon.
		 \end{align}
	 This proves that $F(C_{\epsilon})\subseteq \mathcal{U}_{L\epsilon}(F(K))$.

			To prove the other inclusion
		we fix $v \in K$.
		By \cref{item:intermediate3t} of \cref{thm:appsel}
		we have  
		\begin{equation}
				 \mathrm{dist}_{H}(\pi_2(A_\epsilon), F(K))\leq \epsilon.
		\end{equation} Hence, there exists $(x,y) \in A_{\epsilon}$ such that 
		$\norm{F(v) - \pi_2((x , y)) } 
		=
		\norm{F(v) - y }\leq \epsilon$.
		We observe that for $\epsilon \in (0,\rho)$ we have $x \in C_{\epsilon} \subseteq
		\mathcal{U}_{\epsilon}(K) \subseteq B_{\R^n}(2\rho)$.
		To conclude the proof we estimate 
		$\norm{F(v) - F(x)}$.
		Since $A_{\epsilon} \subseteq \mathcal{U}_{\epsilon} (S)$, we find
		$(w , F(w) ) \in S$ such that
		$\norm{ (x,y) -  (w , F(w))} \leq \epsilon$, and in particular this implies
		$\norm{y - F(w)} \leq \epsilon$ and 
		$\norm{x - w} \leq \epsilon$.
		Hence for $\epsilon \in (0,\rho)$ we have
		\begin{align}
			\norm{F(v) - F(x)} 
			& \leq 
			\norm{F(v) - y} 
			+ \norm{y - F(w)} 
			+ 
			\norm{F(w) - F(x)} \\
			& \leq
		( 2 + \mathrm{Lip}(F , B_{\R^n}(2\rho)) )
		\epsilon=L(F, \rho)\epsilon.
		\end{align}
	This proves that $F(K)\subseteq \mathcal{U}_{L\epsilon}(F(C_\epsilon))$.
		\item Let now $e\in \mathbb{N}$ with $e\leq n$ and $E \simeq \R^e\subseteq \R^n$.  Observe first that $E\cap C_\epsilon$ is the projection on  $E\simeq \R^e$ of $(E \times \R^\ell)\cap A_\epsilon$.  

Note that $D((E\times \R^\ell)\cap A_\epsilon)=(e+\ell,\kappa, \kappa d)$, where $\kappa=\kappa(c)$ is the constant of \cref{item:intermediate4t} from \cref{thm:appsel}. Therefore, the 
		Thom--Milnor bound \eqref{eq:milnorbound} yields
		\[
		b_0(E\cap C_\epsilon)\leq b_0((E \times \R^\ell)\cap A_\epsilon)\leq \beta_{\mathrm{tm}}(2\kappa)(\kappa d)^{e+\ell} \leq \beta(c, d, \ell)^e,
		\]
		for suitable $\beta=\beta(c,d,\ell)>1$.
		\end{enumerate}
		The proof is concluded.
		\end{proof}
		\section*{Declaration}
		\subsection*{Data sharing} Data sharing not applicable to this article as no datasets were generated or analysed during the current study.
\subsection*{Conflict of interest} On behalf of all authors, the corresponding author states that there is no conflict of interest.
	
	\bibliographystyle{alphaabbr}
	\bibliography{selection}

\begin{thebibliography}{BCR98}

\bibitem[BCR98]{BCR}
J.~Bochnak, M.~Coste, and M.-F. Roy.
\newblock {\em G\'eom\'etrie alg\'ebrique r\'eelle (Second edition in english:
  Real Algebraic Geometry)}, volume 12 (36) of {\em Ergebnisse der Mathematik
  und ihrer Grenzgebiete [Results in Mathematics and Related Areas]}.
\newblock Springer-Verlag, Berlin, 1987 (1998).

\bibitem[BD07]{browndavenport}
C.~W. Brown and J.~H. Davenport.
\newblock The complexity of quantifier elimination and cylindrical algebraic
  decomposition.
\newblock In {\em I{SSAC} 2007}, pages 54--60. ACM, New York, 2007.

\bibitem[BL23]{BasuLerario}
S.~Basu and A.~Lerario.
\newblock Hausdorff approximations and volume of tubes of singular algebraic
  sets.
\newblock {\em Math. Ann.}, 387(1-2):79--109, 2023.

\bibitem[BPR06]{BasuPoRoyBook}
S.~Basu, R.~Pollack, and M.-F. Roy.
\newblock {\em Algorithms in real algebraic geometry}, volume~10 of {\em
  Algorithms and Computation in Mathematics}.
\newblock Springer-Verlag, Berlin, second edition, 2006.

\bibitem[BR14]{BR}
S.~Basu and M.-F. Roy.
\newblock Divide and conquer roadmap for algebraic sets.
\newblock {\em Discrete Comput. Geom.}, 52(2):278--343, 2014.

\bibitem[DH88]{davenportheintz}
J.~H. Davenport and J.~Heintz.
\newblock Real quantifier elimination is doubly exponential.
\newblock {\em J. Symbolic Comput.}, 5(1-2):29--35, 1988.

\bibitem[GV09]{gavo}
A.~Gabrielov and N.~Vorobjov.
\newblock Approximation of definable sets by compact families, and upper bounds
  on homotopy and homology.
\newblock {\em J. Lond. Math. Soc. (2)}, 80(1):35--54, 2009.

\bibitem[LRT24]{LRT-Sard}
A.~Lerario, L.~Rizzi, and D.~Tiberio.
\newblock Sard properties for polynomial maps in infinite dimension.
\newblock {\em arXiv 2407.02296}, 2024.

\bibitem[Mat73]{Mather}
J.~N. Mather.
\newblock Generic projections.
\newblock {\em Ann. of Math. (2)}, 98:226--245, 1973.

\bibitem[Vit55]{Vit1}
A.~G. Vitu\v{s}kin.
\newblock {\em O mnogomernyh variaciyah}.
\newblock Gosudarstv. Izdat. Tehn.-Teor. Lit., Moscow, 1955.

\bibitem[Wal78]{Walker}
R.~J. Walker.
\newblock {\em Algebraic curves}.
\newblock Springer-Verlag, New York-Heidelberg, 1978.
\newblock Reprint of the 1950 edition.

\bibitem[YC04]{ComteYomdin}
Y.~Yomdin and G.~Comte.
\newblock {\em Tame geometry with application in smooth analysis}, volume 1834
  of {\em Lecture Notes in Mathematics}.
\newblock Springer-Verlag, Berlin, 2004.

\end{thebibliography}
\end{document}